\documentclass[a4paper,11pt]{amsart}

\usepackage{latexsym}
\usepackage{amssymb,amsmath,enumerate}
\usepackage{graphics}
\usepackage{url}
\usepackage[vcentermath]{youngtab}
  
\usepackage{booktabs}  
\usepackage{caption}
\usepackage{tikz}
\usetikzlibrary{snakes}
\usetikzlibrary{arrows}
\tikzstyle{empty}=[circle,draw=black!80,thick]
\tikzstyle{emptyn}=[circle,draw=black!80,fill=white,scale=0.5] 
\tikzstyle{nero}=[circle,draw=black!80,fill=black!80,thick]

\newtheorem{thm}{Theorem}[section]
\newtheorem{corollary}[thm]{Corollary}

\newtheorem{proposition}[thm]{Proposition}

\newtheorem{lemma}[thm]{Lemma}

\theoremstyle{definition}
\newtheorem{definition}[thm]{Definition}
\newtheorem{remark}[thm]{Remark}
\newtheorem{example}[thm]{Example}

\newtheorem {Step}{Step}

\newtheorem*{thmA}{Theorem A}
\newtheorem*{thmB}{Theorem B}
\newtheorem*{thmC}{Theorem C}
\newtheorem*{CorD}{Corollary D}
\newcommand{\N}{\mathbf{N}}

\renewcommand{\epsilon}{\varepsilon}
\renewcommand{\emptyset}{\varnothing}

\DeclareMathOperator{\GL}{GL}

\newcommand{\fS}{{\mathfrak{S}}}
\newcommand{\nor}{\unlhd}
\def\norm#1#2{{\bf N}_{#1}(#2)}
\def\irr#1{{\rm Irr}(#1)}
\def\irrq#1#2{{\rm Irr}_{#2'}(#1)}
\def\syl#1#2{{\rm Syl}_#1(#2)}

\newcounter{thmlistcnt}
	{\setcounter{thmlistcnt}{0}%
	\begin{list}{\emph{(\roman{thmlistcnt})}}{%
		\usecounter{thmlistcnt}%
		\setlength{\topsep}{0pt}%
		\setlength{\leftmargin}{27pt}%
		\setlength{\itemsep}{0pt}%
		\setlength{\labelwidth}{20pt}
		\setlength{\itemindent}{0pt}}%
	}%
	{\end{list}}%

\numberwithin{equation}{section}

\newcommand{\Syl}{\operatorname{Syl}}

\newcommand{\Aut}{\operatorname{Aut}}

\newcommand{\Irr}{\operatorname{Irr}}

\newcommand{\GU}{\operatorname{GU}}

\newcommand{\FF}{\mathbb{F}}

\def\irrp#1{{\rm Irr}_{p'}(#1)}

\def\irr#1{{\rm Irr}(#1)}

\def\syl#1#2{{\rm Syl}_#1(#2)}
\def\nor{\triangleleft\,}

\def\det#1{{\rm det}(#1)}
\def\ker#1{{\rm ker}(#1)}
\def\norm#1#2{{\bf N}_{#1}(#2)}

   \def \mod#1{\, {\rm mod} \, #1 \, }

\newcommand{\Inn}{{\mathrm {Inn}}}

\newcommand{\diag}{{\mathrm {diag}}}

\newcommand{\Gal}{{\it Gal}}

\newcommand{\CC}{{\mathbb C}}
\newcommand{\CB}{{\mathbf C}}

\newcommand{\QQ}{{\mathbb Q}}
\newcommand{\ZZ}{{\mathbb Z}}
\newcommand{\NB}{{\mathbf N}}

\newcommand{\HC}{{\mathcal H}}

\newcommand{\eps}{\epsilon}

\newcommand{\al}{\alpha}

\newcommand{\gam}{\gamma}
\newcommand{\lam}{\lambda}

\newcommand{\sm}{\sigma}
\newcommand{\usm}{\underline{\sigma}}
\newcommand{\ulam}{\underline{\lambda}}
\newcommand{\vt}{\vdash_{3'}}

\newcommand{\tw}[1]{{}^#1\!}
\renewcommand{\mod}{\bmod \,}

\title[]{Irreducible characters of $3'$-degree of finite symmetric, general linear and unitary groups}
\author{Eugenio Giannelli, Joan Tent, and Pham Huu Tiep}
\address[E.~Giannelli]{Trinity Hall, University of Cambridge, Trinity Lane, CB21TJ, UK}
\email{eg513@cam.ac.uk}
\address[J.~Tent]{Departament de Matem\`atiques, Universitat de Val\`encia, 46100 Burjassot (Val\`encia), Spain}
\email{joan.tent@uv.es}
\address[P. H. Tiep]{Department of Mathematics, University of Arizona, Tucson, AZ 85721, USA}
\email{tiep@math.arizona.edu}

\thanks{The first author gratefully acknowledges financial support by the
ERC Advanced Grant 291512 and by Trinity Hall, Cambridge. 
The second author 
has been supported by 
MTM2014-53810-C2-01 of the Spanish MEyC and by the Prometeo/Generalitat Valenciana.  
The third author gratefully acknowledges the support of the NSF (grant DMS-1201374) and a 
Clay Senior Scholarship.}
\thanks{Part of the paper was written while the first and third author were visiting the Centre Interfacultaire Bernoulli, EPFL, Lausanne, Switzerland. 
It is a pleasure to thank the Clay Mathematics Institute for financial support and the EPFL for 
generous hospitality and stimulating environment.
Another part of the paper was written while the second author was visiting the University of Kaiserslautern, and it is a pleasure to thank Gunter Malle for supporting the research stay. 
Finally we thank Gabriel Navarro for several inspiring conversations on the topic.
}

\begin{document}


\begin{abstract}
Let $G$ be a finite symmetric, general linear, or general unitary group defined over a field of characteristic coprime to $3$.
We construct a canonical correspondence between irreducible characters of degree coprime to $3$ of $G$ and those of $\NB_{G}(P)$, where $P$ is a Sylow $3$-subgroup of $G$. 
Since our bijections commute with the action of the absolute Galois group over the
rationals,  we conclude that fields of values of character
correspondents  are the same. 
\end{abstract}

\maketitle
\thispagestyle{empty}

\section{Introduction}
For any finite group $G$ and any prime number $p$, let $\Irr_{p'}(G)$ denote the set of complex irreducible characters of $G$ of
degree coprime to $p$. 
The {\it McKay conjecture} (first stated in \cite{McK}) asserts that the number $|\mathrm{Irr}_{p'}(G)|$
is equal to $|\mathrm{Irr}_{p'}(\NB_G(P))|$, where $P$ is a Sylow $p$-subgroup of $G$. 
When $p=2$ this conjecture was recently proved in \cite{MalleSpath}. 
Sometimes, it is not only possible to show that 
$|\mathrm{Irr}_{p'}(G)|=|\mathrm{Irr}_{p'}(\NB_G(P))|$, but we can also establish a canonical bijection between the two sets of characters.
In general, a {\it natural} correspondence of characters between a group $G$ and a subgroup $H$ of $G$ does not always exist. 
Particularly rare are the cases where such a correspondence can be found to be compatible with restriction of characters, or equivariant under
the action of Galois or outer automorphisms. The goal of the paper is to construct canonical McKay correspondences for certain important finite groups,
namely the symmetric groups and the general linear and general unitary groups (in characteristic other than $3$), at the prime $p=3$. Note that the validity 
of the McKay conjecture for these groups was already established by Olsson \cite{Olsson}. The novelty of our results lies in the construction of {\it canonical}
McKay bijections for those groups.

A canonical McKay bijection for symmetric groups at $p=2$ was constructed in \cite[Theorem 4.3]{GKNT}, building on \cite[Theorem 3.2]{G}. It is 
natural to ask whether canonical McKay bijections can also exist for other primes. Unfortunately, for $p\geq 5$ the fields of values of the irreducible characters of $\fS_n$ and $\NB_{\fS_n}(P_n)$ are distinct. Hence no canonical bijection can possibly exist when $p \geq 5$, since we would expect it to commute with the action of the absolute Galois group over the rationals (here we denoted by $P_n$ a Sylow $p$-subgroup of $\fS_n$). 
It turns out, somewhat surprisingly, that a canonical McKay bijection does exist for $\fS_n$ at $p=3$. Note that, in general, such a bijection 
does not exist for solvable groups. For instance the fields of values of the $3'$-degree 
irreducible characters of the solvable group $G:=\GL_2(3)$ and $\NB_{G}(P)$ are distinct, for $P$ a Sylow $3$-subgroup of $G$.

\medskip
Our first main result constructs a canonical McKay bijection in the case of symmetric groups $\fS_n$ for the prime $p=3$.

\begin{thmA}
\textit{Let $n$ be any positive integer and let $P \in \Syl_3(\fS_n)$. Then there is a canonical bijection between 
the set $\Irr_{3'}(\fS_n)$ of complex irreducible characters of $3'$-degree of $\fS_n$ and that of 
$\NB_{\fS_n}(P)$.}
\end{thmA}

Our proof relies on previous results on combinatorics of representations of symmetric groups, particularly \cite{Olsson}, and also on the following key step analyzing the case $n=3^k$ for any $k\in\mathbb{N}$. In Theorem \ref{t:2main correspondence p^k} we find a canonical bijection $\chi\mapsto\chi^*$ between $\mathrm{Irr}_{3'}(\fS_{3^k})$ and $\mathrm{Irr}_{3'}(\NB_{\fS_{3^k}}(P_{3^k}))$, that is {\it compatible with character restriction} (i.e.
$\chi^*$ is an irreducible constituent of $\chi\downarrow_{\NB_{\fS_{3^k}}(P_{3^k})}$, for every $\chi\in\mathrm{Irr}_{3'}(\fS_{3^k})$).
This is done by considering the representation theory of the intermediate subgroup $\fS_{3^{k-1}}\wr \fS_3$, lying between $\fS_{3^k}$ and $\NB_{\fS_{3^k}}(P_{3^k})$. 
In particular we can prove the following general theorem, valid for every odd prime number, that we believe is of independent interest. 

\begin{thmB}
\textit{Let $p$ be an odd prime and let $H=\fS_{p^{k-1}}\wr \fS_p\leq \fS_{p^k}$, for some natural number $k$. 
Let $\chi\in\mathrm{Irr}_{p'}(\fS_{p^k})$, then 
$\chi\downarrow_H=\chi^*+\Delta$, where 
$\chi^*\in\mathrm{Irr}_{p'}(\fS_{p^{k-1}}\wr \fS_p)$ and $\Delta$ is a sum (possibly empty) of irreducible characters of degree divisible by $p$. Moreover, the map $\chi\mapsto\chi^*$ is a bijection between 
$\mathrm{Irr}_{p'}(\fS_{p^k})$ and $\mathrm{Irr}_{p'}(\fS_{p^{k-1}}\wr \fS_p)$.}
\end{thmB}

\medskip
The second main result of the paper constructs canonical McKay bijections at $p=3$ for $\GL_n(q)$ and $\GU_n(q)$ when $3 \nmid q$,
building on Theorem A and some results of \cite{Olsson} and \cite{FS}.


\begin{thmC}
\textit{Let $n$ be any positive integer and let $q=p^a$ be any power of a prime $p \neq 3$.  Then there is a canonical bijection between 
the set of complex irreducible characters of $3'$-degree of $\GL_n(q)$ and that of 
$\NB_{\GL_n(q)}(P)$ for $P \in \Syl_3(\GL_n(q))$. Similarly,  there is a canonical bijection between 
the set of complex irreducible characters of $3'$-degree of $\GU_n(q)$ and that of 
$\NB_{\GU_n(q)}(P)$ for $P \in \Syl_3(\GU_n(q))$.}
\end{thmC}

We will prove that the canonical bijections constructed in Theorem C commute with the action of the absolute Galois group $\Gal(\bar\QQ/\QQ)$. 
This implies, for instance, the following corollary.

\begin{CorD}
\textit{Let $n$ be any positive integer and let $q=p^a$ be any power of a prime $p \neq 3$. Let $G$ be either $\GL_n(q)$ or $\GU_n(q)$, and let 
$P$ be a Sylow $3$-subgroup of $G$. Then the fields of values of the $3'$-degree complex irreducible characters of $G$ 
are equal to the fields  of values of the $3'$-degree complex irreducible characters of $\NB_{G}(P)$.}
\end{CorD}

\section{Preliminaries}

\subsection{Extensions of characters and wreath products}\label{pre_ch}

In this section we will recall some known results on extensions of characters in 
wreath products and other character theoric results  which will be needed in the sequel. 

Let $\fS_m$ be the symmetric group of degree $m$.  
If $K$ is a finite group, write $K^m=\underbrace{K\times\cdots\times K}_m$
for the $m$-fold external direct product of $K$.  
The natural action of $\fS_m$ on the direct factors of $K^m$
induces an action via automorphisms of $\fS_m$ on $K^m$, so we can  define the wreath product $H=K\wr \fS_m:=K^m\rtimes \fS_m$. 
As it is customary, we denote the elements of $H$ by
$(x_1,\ldots, x_n; y)$, where $x_i\in K$ and $y\in \fS_m$. 
Also, we will identify subgroups and elements of both the base group $K^m$ of $H$, 
and the permutation group $\fS_m$ acting on the factors of $K^m$, with their natural embeddings in $H$. 

With the same notation as above,
recall that the irreducible characters of $K^m$ are of the form
$$
\psi=\psi_1\otimes\cdots\otimes\psi_m\, ,
$$ where $\psi_i\in\irr{K}$ for each $i$, and $\otimes$ denotes the external direct product of characters. 
In particular, note that $\psi\in\irr{K^m}$  is invariant in $H$ if and only if
$$
\psi=\psi_1^{\otimes m}:=\underbrace{\psi_1\otimes\psi_1\otimes\cdots\otimes\psi_1}_m \, 
$$for a uniquely determined $\psi_1\in\irr K$. In this case, $\psi$ has a distinguished \color{black}
irreducible extension  to $H$ defined by the following statement.

\begin{lemma}\label{wr-ext}
Let $K$ be any finite group with an irreducible $\CC K$-module $U$. For any $m \in \ZZ_{\geq 2}$, consider 
$H = K \wr \fS_m = K^m \rtimes \fS_m$. 
Let $t$ be any transposition in the standard subgroup $\fS_m$ of $H$.
Then the irreducible module $U^{\otimes m} = U \otimes \ldots \otimes U$ of $K^m \lhd H$ has a unique extension
$V$ to $H$, with character say $\theta$ satisfying the following conditions:

\begin{enumerate}[\rm(i)]
\item If $2 \nmid \dim(U)$ then $\det{\theta}|_{\fS_m}$ is trivial.

\item If $2|\dim(U)$ but $m \geq 3$, then $\det{\theta}|_{\fS_m}$ is trivial and $\theta(t) \in \ZZ_{> 0}$.

\item If $2|\dim(U)$ and $m = 2$, then $\theta(t) \in \ZZ_{> 0}$.
 
\end{enumerate}
In particular, if $K^m \leq L \leq H$ and $\gcd(|K|,|L/K^m|) = 1$, then $V|_L$ is the canonical extension in 
Theorem \ref{canonical_ext} (below). Moreover, if a finite group $A$ acts on $K$ and we extend its action to
$H$ by letting $A$ act trivially on $\fS_m$, then the map $U \mapsto V$ is $A$-equivariant. Furthermore,
the map $U \mapsto V$ is $\Gal(\bar\QQ/\QQ)$-equivariant.
\end{lemma}

\begin{proof}
First, fix a basis $(e_i \mid 1 \leq i \leq d)$ of the space $U$. Then we let any element $g = (h_1, \ldots ,h_m) \in K^m$ act on 
$U^{\otimes m}$ via 
$$g: e_{i_1} \otimes e_{i_2} \otimes \ldots \otimes e_{i_m} \mapsto h_1(e_{i_1}) \otimes h_2(e_{i_2}) \otimes \ldots \otimes h_m(e_{i_m}),~~
1 \leq i_1, \ldots,i_m \leq d.$$ 
Then we can define the action of any element $\pi$ in the natural subgroup $\fS_m$ on $U^{\otimes m}$ via
$$\pi:e_{i_1} \otimes e_{i_2} \otimes \ldots \otimes e_{i_m} \mapsto e_{i_{\pi(1)}} \otimes e_{i_{\pi(2)}} \otimes \ldots \otimes e_{i_{\pi(m)}},~~
1 \leq i_1, \ldots,i_m \leq d.$$  
One can check that this turns $U^{\otimes m}$ into an $H$-module which we denote by $V_1$. Note that the trace of $t$ on $V_1$ is 
$d^{m-1} > 0$. By Gallagher's theorem  \cite[Corollary 6.17]{isaacs}, if $V_2 \not\cong V_1$ is another extension of $U^{\otimes m}$ to $H$, then
$V_2 \cong V_1 \otimes W$, where $W$ is the sign $\CC \fS_m$-module. In particular, the trace of $t$ on $V_2$ is $-d^{m-1} < 0$, and 
the claim in (iii) follows, taking $V = V_1$. 

Let $\theta_i$ be the character afforded by $V_i$, $i = 1,2$. Note that the action of $t$ on the basis vectors of $U^{\otimes m}$ 
gives rise to a disjoint product of $d^{m-1}(d-1)/2$ $2$-cycles. Hence $\det{\theta_1}(t) = 1$ in the case of (ii), and we again done here.
Assume $2 \nmid d$. Then $\det{\theta_2}(t) = \det{\theta_1}(t)(-1)^{d^m} = -\det{\theta_1}(t)$, and so there is a unique $i \in \{1,2\}$
such that $\det{\theta_i}(t) = 1$, whence we are done in (i) as $\fS_m$ is generated by transpositions.

Assume furthermore that $K^m < L \leq H$ and $\gcd(|K|,|L/L^m|)=1$. Then we are not in case (iii), and so $o(\theta|_L) = o(\theta|_{K^m})$
by the previous result,
whence $\theta|_L$ is the canonical extension singled out in Theorem \ref{canonical_ext}.

It is also straightforward to check that the map $U \mapsto V$ is equivariant under the action of both $A$ and $\Gal(\bar\QQ/\QQ)$.
\end{proof}

We remark that the extension described in the previous lemma
is uniquely determined for a fixed complement $\fS_m$ of $K^m$ in $H$
as in the statement, but that different choices of such a complement may 
produce different extensions of the module $U^{\otimes m}$ (differing at most by tensoring by a sign module).
Also, it is clear that conjugate complements of $K^m$ in $H$ inducing the wreath product
$K\wr \fS_m$ would lead to the same extension of  $U^{\otimes m}$
to $H$ described in Lemma \ref{wr-ext}. 

Another situation in which uniquely determined \color{black} extensions of characters exist is  the following. 
We point out that next theorem is a particular case of a slightly more general fact, but we will only need
the result as stated below (see Corollary 8.16 of \cite{isaacs} and the comments before it for details).

\begin{thm}\label{canonical_ext}
Let $N\nor G$ and let $\psi\in\irr{N}$ be $G$-invariant. Suppose that 
$(|G/N|, |N|)=1$. Then $\psi$ has a uniquely determined irreducible canonical extension $\hat\psi\in\irr G$, characterized 
by the property that $\hat\psi$ and $\psi$ have the same determinantal order: $o(\hat\psi) = o(\psi)$.
\end{thm}



Suppose that $N\nor G$ and that $\psi\in\irr{N}$ is extendible to $G$. 
Then recall that Gallagher's Corollary 6.17 of \cite{isaacs} provides a complete description
of the set of irreducible characters $\irr{G\, \, |\, \psi}$ of $G$ lying over $\psi$, in terms of 
the characters of $G/N$.


At certain points in our arguments we will also need to use some known correspondences of characters. 
We refer the reader to \cite[Theorem 6.11]{isaacs} for a reference on the standard Clifford's
correspondence. 
We next include for future reference two useful results which are contained in M. Isaacs'  work \cite[Corollaries 4.2 and 4.3]{bp}.

\begin{proposition}\label{isa_lemma}
Let $N\nor G$ and $K\leq G$ with $NK = G$ and
$N\cap K = M$. 

\begin{enumerate}
\item \label{restriction}{Suppose that $\theta \in \irr{N}$ is invariant in $G$ and assume $\varphi=\theta_M$ is
irreducible. Then restriction of characters defines a bijection from $\irr{G\, |\, \theta}$ into $\irr{K \, |\, \varphi}$.} 
\item \label{induction}{Suppose that $\varphi\in\irr {M}$ is invariant in $K$ and assume $\theta=\varphi^N$ is
irreducible. Then induction of characters defines a bijection from $\irr{K \, |\, \varphi}$ into $\irr{G\, |\, \theta}$.}
\end{enumerate}
\end{proposition}

\subsection{Background on combinatorics and representations of $\fS_n$}
In this section we recall some basic facts in the representation theory of symmetric groups. We refer the reader to \cite{James}, \cite{JK} or \cite{OlssonBook} for a more detailed account. 
A partition $\lambda=(\lambda_1,\lambda_2,\dots ,\lambda_\ell)$ is a finite non-increasing sequence of positive integers. We say that $\lambda_i$ is a part of $\lambda$. We call $\ell=\ell(\lambda)$ the length of $\lambda$ and say that $\lambda$ is a partition of $|\lambda|=\sum\lambda_i$. We denote by $\mathcal{P}(n)$ the set consisting of all the partitions of $n$. 
The Young diagram of $\lambda$ is the set $[\lambda]=\{(i,j)\in{\mathbb N}\times{\mathbb N}\mid 1\leq i\leq\ell(\lambda),1\leq j\leq\lambda_i\}$. Here we orient ${\mathbb N}\times{\mathbb N}$ with the $x$-axis pointing right and the $y$-axis pointing down.
Given $\lambda\in\mathcal{P}(n)$, we denote by $\lambda'$ the \textit{conjugate partition} of $\lambda$.

We say that a partition $\mu$ is contained in $\lambda$, written $\mu\subseteq\lambda$, if $\mu_i\leq\lambda_i$, for all $i\geq1$. When this occurs, we call the non-negative sequence $\lambda\smallsetminus\mu=(\lambda_i-\mu_i)_{i=1}^\infty$ a skew-partition, and we call the diagram
$[\lambda\smallsetminus\mu]=\{(i,j)\in{\mathbb N}\times{\mathbb N}\mid 1\leq i\leq\ell(\lambda),\mu_i<j\leq\lambda_i\}$ 
a skew Young diagram.

The rim of $[\lambda]$ is the subset of nodes ${\mathcal R}(\lambda)=\{(i,j)\in[\lambda]\mid (i+1,j+1)\not\in[\lambda]\}$. Given $(r,c)\in[\lambda]$, the associated rim-hook is $h(r,c)=\{(i,j)\in {\mathcal R}(\lambda)\mid r\leq i,c\leq j\}$.
Then $h=h(r,c)$ contains $e:=\lambda_r-r+\lambda'_c-c+1$ nodes, in $a(h)=\lambda_r-c+1$ columns
and $\lambda'_c-r+1$ rows. We call $\mathrm{leg}(h)=\lambda'_c-r$ the leg-length of $h$. We refer to $h$ as an $e$-hook of $\lambda$. The integer $e$ is sometimes denoted as $|h|$. Removing $h$ from $[\lambda]$ gives the Young diagram of a partition denoted $\lambda- h$. In particular $|\lambda- h|=|\lambda|-e$ and $h$ is a skew Young diagram. 

Let $h$ be an $e$ rim-hook which has leg-length $\ell$. 
The associated hook partition of $e$ is $\hat h=(e-\ell,1^\ell)$. So $(e-\ell,1^\ell)$ coincides with its $(1,1)$ rim-hook. 
Given any natural number $n$, we denote by $\mathcal{H}(n)$ the subset of $\mathcal{P}(n)$ consisting of all \textit{hook partitions} of $n$. 
Namely $\mathcal{H}(n)=\{(n-x,1^x)\ |\ 0\leq x\leq n-1\}$.

Let $p$ be a (not necessarily prime) natural number, we say that a partition $\gamma$ is a $p$-core if it does not have removable $p$-hooks. Given a partition $\lambda$ of $n$, its $p$-core $\lambda_{(p)}$ is the partition obtained by successively removing $p$-hooks from $\lambda$. 
We will denote by $\lambda^{(p)}=(\lambda^0,\lambda^1,\ldots,\lambda^{p-1})$ the $p$-quotient of $\lambda$ (see \cite[Section 3]{OlssonBook} for the definition). Following Olsson's convention we always consider abaci consisting of a multiple of $p$ number of beads. This assumption guarantees that the $p$-quotient of a partition is well defined (see either \cite[Chapter 2]{JK} or again \cite[Section 3]{OlssonBook} for the definition of James' abacus).

For any given $\lambda\in\mathcal{P}(n)$ we denote by $T(\lambda)$ its $p$-core tower (see \cite[Section 6]{OlssonBook}). In particular we find convenient to think of the $p$-core tower as a sequence $T(\lambda)=(T_{j}(\lambda))_{j=0}^{\infty}$, 
where the $j$-th layer (or row) $T_j(\lambda)$ consists of $p^j$ $p$-core partitions. We denote by $|T_j(\lambda)|$ the sum of the sizes of the $p$-cores in the $j$-th layer of $T(\lambda)$. 
We have that $n=\sum_j|T_j(\lambda)|p^j$. Moreover, every partition $\lambda$ of $n$ is uniquely determined by its $p$-core tower. This follows by repeated applications of  \cite[Proposition 3.7]{OlssonBook}. 

There is a natural one-to-one correspondence between irreducible characters of $\fS_n$ and partitions of $n$. We denote by $\chi^\lambda$ the irreducible character labelled by $\lambda\in\mathcal{P}(n)$. 
The Murnaghan-Nakayama rule (see \cite[Page 79]{James}) allows to explicitly compute the entire character table of $\fS_n$.
A useful corollary of this rule is the so called hook-length formula (see \cite[Chapter 20]{James}). This gives a closed fowmula for the degree of any irreducible character of $\fS_n$. In particular, for a hook partition 
$h=(n-x,1^x)\in\mathcal{H}(n)$ we have that $\chi^h(1)={n-1\choose x}$.

Let now $p$ be a prime number. Given any partition $\lambda$ of $n$, the $p$-core tower of $\lambda$ encodes all the information concerning the $p$-part of the degree of the corresponding irreducible character $\chi^\lambda(1)$. In particular, in \cite{Mac} the following fundamental result is proved. 

\begin{thm}\label{Mac}
Let $p$ be a prime number and let $\lambda$ be a partition of $n\in\mathbb{N}$.
Suppose that $n=\sum_{j=0}^ka_jp^j$ is the $p$-adic expansion of $n$. Then $\chi^\lambda\in\mathrm{irr}_{p'}(\fS_n)$ if and only if $|T_j(\lambda)|=a_j$ for all $j\in\mathbb{N}$.
\end{thm}

\subsection{Some remarks on the Littlewood-Richardson rule}

In order to prove some of the main results in this article, we will make extensive use of a generalised version of the Littlewood-Richardson rule. 
This is probably known to experts, but we were not able to find an appropriate reference in the literature. 
For the reader's convenience we start by recalling the classic Littlewood-Richardson rule. 

\begin{definition}
Let ${\mathcal A}=a_1,\dots,a_k$ be a sequence of positive integers. The type of $\mathcal{A}$ is the sequence of non-negative integers $m_1,m_2,\dots$ where $m_i$ is the number of occurrences of $i$ in $a_1,\dots,a_k$. We say that $\mathcal{A}$ is a \textit{reverse lattice sequence} if the type of its prefix $a_1,\dots,a_j$ is a partition, for all $j\geq1$. Equivalently, for each $j=1,\dots,k$ and $i\geq2$
$$
|\{u\mid 1\leq u\leq j, a_u=i-1\}|\geq |\{v\mid 1\leq v\leq j, a_v=i\}|.
$$
\end{definition}

Let $\alpha\vdash n$ and $\beta\vdash m$ be partitions. The outer tensor product $\chi^\alpha\otimes\chi^\beta$ is an irreducible character of ${\mathfrak S}_n\times{\mathfrak S}_m$. Inducing this character to ${\mathfrak S}_{n+m}$ we may write
$$
(\chi^\alpha\otimes\chi^\beta){\uparrow^{{\mathfrak S}_{n+m}}}=\sum_{\gamma\vdash(n+m)}C_{\alpha,\beta}^\gamma\chi^\gamma.
$$
The \textit{Littlewood-Richardson rule} asserts that $C_{\alpha,\beta}^\gamma$ is zero if $\alpha\not\subseteq\gamma$ and otherwise equals the number of ways to replace the nodes of the diagram $[\gamma\smallsetminus\alpha]$ by natural numbers such that 
\begin{enumerate}
\item The numbers are weakly increasing along rows.
\item The numbers are strictly increasing down the columns.
\item The sequence obtained by reading the numbers from right to left and top to bottom is a reverse lattice sequence of type $\beta$.
\end{enumerate}
We call any such configuration a \textit{Littlewood-Richardson configuration of} $[\gamma\smallsetminus\alpha]$.

\medskip

The Littlewood-Richardson rule describes the decomposition of the restriction of any irreducible character of $\fS_n$ to any $2$-fold Young subgroup of $\fS_n$.
In this article we will need to have some control on the restriction of irreducible characters to arbitrary Young subgroups of $\fS_n$. 
We start by introducing some notation. 
Let $k\in\mathbb{N}_{\geq 2}$ and let $n_i\in\mathbb{N}_{\geq 1}$ for all $i\in\{1,\ldots, k\}$. Let $\rho$ be a partition of $n=n_1+n_2+\cdots +n_k$ and let $\mu_i$ be a partition of $n_i$ for all $i\in\{1,\ldots, k\}$. Denote by $\underline{\mu}$ the sequence $\underline{\mu}=(\mu_1,\ldots, \mu_k)$.

\begin{definition}\label{def:Amuro}
For all $j\in\{1,\ldots, k\}$ let $m_j=n_1+\cdots +n_j$.
We denote by $\mathcal{A}_{\underline{\mu}}^\rho$ the subset of $\mathcal{P}(m_1)\times \mathcal{P}(m_2)\times\cdots\times\mathcal{P}(m_k)$ consisting of all the sequences of partitions $(\lambda_1,\lambda_2,\ldots,\lambda_k)$ such that: 
\begin{enumerate}
\item[(i)] $\mu_1=\lambda_1\subseteq \lambda_2\subseteq\cdots\subseteq\lambda_k=\rho$.
\item[(ii)] $\lambda_{j+1}\smallsetminus\lambda_j$ admits a Littlewood-Richardson configuration of type $\mu_{j+1}$, for all $j\in\{1,\ldots, k-1\}$.
\end{enumerate}

Notice that condition (ii) is equivalent to say that the Littlewood-Richardson coefficient $$C_{\lambda_j\mu_{j+1}}^{\lambda_{j+1}}=\left\langle (\chi^{\lambda_j}\otimes\chi^{\mu_{j+1}})
\big\uparrow_{\fS_{m_j}\times\fS_{n_{j+1}}}^{\fS_{m_{j+1}}}, \chi^{\lambda_{j+1}}\right\rangle\neq 0, \ \ \ \text{for all $j$}.$$
\end{definition}

\begin{remark}\label{rmk:Amurosigma}
Let $\underline{\mu}=(\mu_1,\ldots, \mu_k)$ and $\rho$ be as above. 
Denote by $\underline{\mu}^\star=(\mu_1,\ldots, \mu_{k-1})$ the sequence obtained by removing the last component from $\underline{\mu}$. 
For $\sigma\in\mathcal{P}(n-n_k)$ we denote by $\mathcal{A}_{\underline{\mu}}^\rho(\sigma)$ the subset of 
$\mathcal{A}_{\underline{\mu}}^\rho$ defined by 
$$\mathcal{A}_{\underline{\mu}}^\rho(\sigma)=
\{\underline{\lambda}\in\mathcal{A}_{\underline{\mu}}^\rho\ :\ \lambda_{k-1}=\sigma\}.$$
By definition we have that $\mathcal{A}_{\underline{\mu}}^\rho(\sigma)\neq \emptyset$ if and only if $\mathcal{A}_{\underline{\mu}^\star}^\sigma\neq \emptyset$ and 
$C_{\sigma\mu_k}^\rho\neq 0$.
Moreover, for every $\sigma\in\mathcal{P}(n-n_k)$ such that
$C_{\sigma\mu_k}^\rho\neq 0$, the map 
$$(\lambda_1,\ldots, \lambda_{k-2},\sigma)\mapsto (\lambda_1,\ldots, \lambda_{k-2},\sigma,\rho),$$
is a bijection between $\mathcal{A}_{\underline{\mu}^\star}^\sigma$ and 
$\mathcal{A}_{\underline{\mu}}^\rho(\sigma)$.
\end{remark}

\begin{lemma}\label{lem:neq0}
Keeping the notation introduced above, we have that 
$$\left\langle (\chi^{\mu_1}\otimes\chi^{\mu_2}\otimes\cdots\otimes\chi^{\mu_k})
\big\uparrow_{\fS_{n_1}\times\cdots\times\fS_{n_{k}}}^{\fS_{n}}, \chi^{\rho}\right\rangle\neq 0,$$
if and only if $\mathcal{A}_{\underline{\mu}}^\rho\neq\emptyset$.
\end{lemma}
\begin{proof}
We proceed by induction on $k$. 
If $k=2$ then the statement follows from the Littlewood-Richardson rule.

Assume that $\mathcal{A}_{\underline{\mu}}^\rho\neq\emptyset$. 
Let $k\geq 3$ and let $\underline{\lambda}=(\lambda_1,\ldots,\lambda_k)\in\mathcal{A}_{\underline{\mu}}^\rho$. 
From Definition \ref{def:Amuro} we have that 
$$\left\langle (\chi^{\lambda_{k-1}}\otimes\chi^{\mu_{k}})
\big\uparrow_{\fS_{n-n_k}\times\fS_{n_{k}}}^{\fS_{n}}, \chi^{\rho}\right\rangle\neq 0.$$
Moreover, the sequence $(\lambda_1,\lambda_2,\ldots,\lambda_{k-1})\in\mathcal{A}_{\underline{\mu}^\star}^{\lambda_{k-1}},$
where $\underline{\mu}^\star:=(\mu_1,\ldots, \mu_{k-1})$. 
By inductive hypothesis we deduce that 
$$\left\langle (\chi^{\mu_1}\otimes\chi^{\mu_2}\otimes\cdots\otimes\chi^{\mu_{k-1}})
\big\uparrow_{\fS_{n_1}\times\cdots\times\fS_{n_{k-1}}}^{\fS_{n-n_k}}, \chi^{\lambda_{k-1}}\right\rangle\neq 0.$$
We now let $B=\fS_{n_1}\times\cdots\times\fS_{n_{k}}$, $C=\fS_{n_1}\times\cdots\times\fS_{n_{k-1}}$ and $H=\fS_{n-n_k}\times \fS_{n_k}$.
Moreover, given any sequence of partitions $\underline{\nu}=(\nu_1,\ldots, \nu_s)$ we denote by $\chi^{\underline{\nu}}$ the irreducible character $\chi^{\nu_1}\otimes\cdots\otimes\chi^{\nu_s}$.
We conclude by observing that, 
$$\left\langle \chi^{\underline{\mu}}\big\uparrow_{B}^{\fS_{n}}, \chi^{\rho}\right\rangle= 
\left\langle \big(\chi^{\underline{\mu}^\star}\big\uparrow_{C}^{\fS_{n-n_k}}\otimes\chi^{\mu_{k}}\big)
\big\uparrow_{H}^{\fS_{n}}
, \chi^{\rho}\right\rangle\geq \left\langle (\chi^{\lambda_{k-1}}\otimes\chi^{\mu_{k}})
\big\uparrow_{H}^{\fS_{n}}, \chi^{\rho}\right\rangle\neq 0.$$

\medskip

Assume now that $\left\langle \chi^{\underline{\mu}}\big\uparrow_{B}^{\fS_{n}}, \chi^{\rho}\right\rangle\neq 0$. Since $\chi^{\underline{\mu}}\big\uparrow_{B}^{\fS_{n}}=
\big(\chi^{\underline{\mu}^\star}\big\uparrow_{C}^{\fS_{n-n_k}}\otimes\chi^{\mu_{k}}\big)
\big\uparrow_{H}^{\fS_{n}}$, we deduce that there exists a partition 
$\sigma\in \mathcal{P}(n-n_k)$ such that $\chi^\sigma$ is an irreducible constituent of $\chi^{\underline{\mu}^\star}\big\uparrow_{C}^{\fS_{n-n_k}}$ and such that 
$C_{\sigma\mu_k}^\rho\neq 0$. 
By inductive hypothesis there exists a sequence  $(\lambda_1,\ldots, \lambda_{k-2},\sigma)\in\mathcal{A}_{\underline{\mu}^\star}^{\sigma}\neq \emptyset$. It is now easy to observe that $(\lambda_1,\ldots,\lambda_{k-2}, \sigma,\rho)\in \mathcal{A}_{\underline{\mu}}^\rho$.
\end{proof}

The following proposition is a generalisation of Lemma \ref{lem:neq0}. 
We decided to keep two distinct statements for the convenience of the reader. 
First we need to introduce a last piece of notation. 
For $\underline{\lambda}=(\lambda_1,\ldots,\lambda_k)\in \mathcal{A}_{\underline{\mu}}^\rho$ we denote by $d_{\underline{\lambda}}$ the natural number defined by $$d_{\underline{\lambda}}=
\prod_{j=1}^{k-1}C_{\lambda_j\mu_{j+1}}^{\lambda_{j+1}}.$$
Again for any sequence of partitions $\underline{\nu}=(\nu_1,\ldots, \nu_s)$ we denote by $\chi^{\underline{\nu}}$ the irreducible character $\chi^{\nu_1}\otimes\cdots\otimes\chi^{\nu_s}$.
\begin{lemma}\label{lem:dlambda}
Let $\underline{\mu}=(\mu_1,\ldots,\mu_k)\in\mathcal{P}(n_1)\times\cdots\times\mathcal{P}(n_k)$ and let $\rho\in \mathcal{P}(n)$, where $n=n_1+\cdots +n_k$.    
Then $$\left\langle \chi^{\underline{\mu}}\big\uparrow_{\fS_{n_1}\times\cdots\times\fS_{n_{k}}}^{\fS_{n}}, \chi^{\rho}\right\rangle=
\sum_{\underline{\lambda}\in\mathcal{A}_{\underline{\mu}}^\rho}d_{\underline{\lambda}} .$$
\end{lemma}
\begin{proof}
We proceed by induction on $k\in\mathbb{N}_{\geq 2}$. If $k=2$ the statement coincide with the Littlewood-Richardson rule. 
Suppose that $k\geq 3$ and denote by $B$ and $C$ the Young subgroups of $\fS_n$ defined by 
$$B=\fS_{n_1}\times\cdots\times\fS_{n_k},\ \ \ C=\fS_{n_1}\times\cdots\times\fS_{n_{k-1}}.$$
For any sequence $\underline{\lambda}=(\lambda_1,\ldots,\lambda_k)\in\mathcal{A}_{\underline{\mu}}^\rho$ we have that 
$\underline{\lambda}^\star=(\lambda_1,\ldots,\lambda_{k-1})\in\mathcal{A}_{\underline{\mu}^\star}^{\lambda_{k-1}}$, 
where $\underline{\mu}^\star=(\mu_1,\ldots,\mu_{k-1})$. 
Moreover, by inductive hypothesis we have that 
$$\left\langle \chi^{\underline{\mu}^\star}
\big\uparrow_{C}^{\fS_{n-n_k}}, \chi^{\lambda_{k-1}} \right\rangle=
\sum_{\underline{\tau}\in\mathcal{A}_{\underline{\mu}^\star}^{\lambda_{k-1}}}
d_{\underline{\tau}}.$$
It follows that 
\begin{eqnarray*}
\left\langle\chi^{\underline{\mu}}\big\uparrow_{B}^{\fS_{n}}, \chi^{\rho}\right\rangle
&=& \left\langle \big(\chi^{\underline{\mu}^\star}
\big\uparrow_{C}^{\fS_{n-n_k}}\otimes \chi^{\mu_k}\big)\big\uparrow_{\fS_{n-n_k}\times \fS_{n_k}}^{\fS_n}, \chi^{\rho}\right\rangle\\
&=& \sum_{\sigma\in\mathcal{P}(n-n_k)}\big(
\sum_{\underline{\tau}\in\mathcal{A}_{\underline{\mu}^\star}^{\sigma}}
d_{\underline{\tau}}\big)
\left\langle\chi^{\sigma}\otimes\chi^{\mu_k}\big\uparrow_{\fS_{n-n_k}\times \fS_{n_k}}^{\fS_{n}}, \chi^{\rho}\right\rangle\\
&=& \sum_{\sigma\in\mathcal{P}(n-n_k)}
\big(
\sum_{\underline{\tau}\in\mathcal{A}_{\underline{\mu}^\star}^{\sigma}}
d_{\underline{\tau}}\big)C_{\sigma\mu_k}^\rho\\
&=& \sum_{\sigma\in\mathcal{P}(n-n_k)}
\big(
\sum_{\underline{\lambda}\in\mathcal{A}_{\underline{\mu}}^{\rho}(\sigma)}
d_{\underline{\lambda}}\big)\\
&=& \sum_{\underline{\lambda}\in\mathcal{A}_{\underline{\mu}}^\rho}
d_{\underline{\lambda}},
\end{eqnarray*}
where the last two equalities above follow from the discussion in Remark \ref{rmk:Amurosigma}.
\end{proof}

The key step in the proof of Lemma \ref{lem:1} below is to use Lemma \ref{lem:dlambda} in the specific case where $\rho\in\mathcal{H}(p^k)$  and $\underline{\mu}=\underbrace{(\mu,\mu,\ldots,\mu)}_{p-\text{times}}$, for some $\mu\in\mathcal{P}(p^{k-1})$.

\section{A canonical McKay bijection for $\fS_{3^k}$}

Our next goal is to construct a canonical bijection between $\mathrm{Irr}_{3'}(\fS_{3^k})$ and $\mathrm{Irr}_{3'}(\N_{\fS_{3^k}}(P_{3^k}))$, where $P_{3^k}\in\syl{3}{\fS_{3^k}}$. As mentioned in the introduction, this will be the key step towards the proof of Theorem A.
More precisely, we will first prove Theorem B, by describing a \textit{global} bijection between the irreducible characters of $p'$-degree of $\fS_{p^k}$ and those of the stabilizer $\fS_{p^{k-1}}\wr \fS_p$
of a set partition of $\{ 1,\ldots, p^k\}$ into $p$ sets of size $p^{k-1}$\color{black}. Afterwards we will fix $p=3$ and we will construct a \textit{local} bijection between 
the irreducible characters of $3'$-degree  of $\fS_{3^{k-1}}\wr \fS_3$ and those of the normaliser of a Sylow $3$-subgroup of $\fS_{3^{k}}$. 

\subsection{A global bijection and the proof of Theorem B}
Let $p$ be an odd prime and let $n=p^k$ for some $k\in\mathbb{N}$.
This section is devoted to the proof of Theorem B.  
In particular we show the existence of a natural bijection between $\mathrm{Irr}_{p'}(\fS_{p^k})$ and $\mathrm{Irr}_{p'}(\fS_{p^{k-1}}\wr \fS_p)$, canonically defined by the restriction functor.

We will use the following easy observation, whose proof is omitted: 

\begin{lemma} \label{complement}
Let  $H \leq \fS_{p^k}$ be the stabilizer
of a set partition of $\{ 1,\ldots, p^k\}$ into $p$ sets of size $p^{k-1}$. Then
$H=\fS_{p^{k-1}}\wr X$, where $X\cong\fS_{p}$ is a subgroup of $\fS_{p^k}$
which induces by conjugation  the full permutation group on the
direct factors of $(\fS_{p^{k-1}})^p$.
Furthermore, $X$ is unique up to $H$-conjugacy. 
\end{lemma}
\color{black}

By a slight abuse of notation, we will denote a subgroup $X\cong \fS_p$ of $H$ as in the previous
lemma simply by $\fS_p$. 
Note that the above result implies that if $\chi\in\irr{(\fS_{p^{k-1}})^p}$ is invariant in $H=\fS_{p^{k-1}}\wr \fS_p$, then
Lemma \ref{wr-ext} provides a uniquely defined extension of $\chi$ to $H$ which is independent of the complement of $(\fS_{p^{k-1}})^p$
in $H$ considered. 

\medskip

It is well known that $\mathrm{Irr}_{p'}(\fS_{p^k})=\{\chi^\lambda\ |\ \lambda\in\mathcal{H}(p^k)\}$.  On the other hand, if $\lambda\in\mathcal{H}(p^{k-1})$ and we write $\chi=\chi^\lambda\in\irrp{\fS_{p^{k-1}}}$, 
then 
$\chi^{\otimes p}\in\irrp{(\fS_{p^{k-1}})^p}$
has a uniquely defined extension $\tilde{\chi}\in\irrp{\fS_{p^{k-1}}\wr \fS_p}$ described in Lemma \ref{wr-ext}. 
In particular, by Gallagher's theorem  \cite[Corollary 6.17]{isaacs} for any $\mu\in\mathcal{H}(p)$
we have that $$\chi(\lambda, \mu):=\tilde{\chi}\cdot\chi^\mu\in\irrp{\fS_{p^{k-1}}\wr \fS_p\, |\, \chi^{\otimes p}}\, ,$$ and in fact 

$$\mathrm{Irr}_{p'}(\fS_{p^{k-1}}\wr \fS_p)=\{\chi(\lambda, \mu) |\ \lambda\in\mathcal{H}(p^{k-1}), \mu\in\mathcal{H}(p)\}.$$ We
will use this notation throughout this section. 
\color{black}


We start with a technical lemma, that will be crucial to construct the desired correspondence.

\begin{lemma}\label{lem:1}
Let $j\in\{0,1,\ldots, p^k-1\}$ and let $h_j=(p^k-j,j)\in \mathcal{H}(p^k)$. Suppose that $j=pm+x$ for some unique $m\in\{0,1,\ldots,p^{k-1}-1\}$ and 
$x\in\{0,1,\ldots, p-1\}$. Let $\lambda=(p^{k-1}-m,1^m)\in\mathcal{H}(p^{k-1})$. The following holds: 
\begin{equation*}
\left\langle \chi^{h_j}\downarrow_{(\fS_{p^{k-1}})^p}, (\chi^\mu)^{\otimes p}  \right\rangle=
\begin{cases}
0 &\text{ if } \mu\in\mathcal{P}(p^{k-1})\smallsetminus\{\lambda\}\,,\\
p-1\choose x &\text{ if } \mu=\lambda\,.\\
\end{cases}
\end{equation*} 
\end{lemma}
\begin{proof}
In order to ease the notation we let $K=(\fS_{p^{k-1}})^p$.
Let $\mu$ be a partition of $p^{k-1}$ and suppose that 
$(\chi^\mu)^{\otimes p}$ is a constituent of $\chi^{h_j}\downarrow_K$. An easy consequence of the Littlewood-Richardson rule shows that $\mu$ must be a subpartition of $h_j$. Hence $\mu=(p^{k-1}-a,1^a)\in\mathcal{H}(p^{k-1})$, for some $a\in\{0,1,\ldots,p^{k-1}-1\}$. 

It is now convenient to change point of view. 
By Lemma \ref{lem:neq0}, we observe that $\chi^\rho$ is an irreducible constituent of the induction from $K$ to $\fS_{p^k}$ of $(\chi^\mu)^{\otimes p}$ if and only if there exists a sequence $\mu=\mu_1\subseteq\mu_2\subseteq\ldots\subseteq\mu_p=\rho$ of partitions such that $\mu_i\in\mathcal{P}(ip^{k-1})$ for every $i\in\{1,2,\ldots,p-1\}$, and such that there exists a Littlewood-Richardson configuration of type $\mu$ of  $[\mu_{i+1}\smallsetminus\mu_i]$ for every $i\in\{1,2,\ldots,p-1\}$. In this case we say that $(\mu_1,\mu_2,\ldots,\mu_p)$ is a $\mu$-sequence of $\rho$.

In particular if $\rho\in\mathcal{H}(p^k)$ then we notice that $ap+1\leq \ell(\rho)\leq (a+1)p$. 
This can be seen by observing that $\mu$ has $a+1$ parts and that for all $i\in\{1,\ldots,p-1\}$ we have that $\ell(\mu_{i+1})=\ell(\mu_i)+a+\varepsilon_i$ for some $\varepsilon_i\in\{0,1\}$.
This fact is enough to deduce that if $\left\langle \chi^{h_j}\downarrow_{K}, (\chi^\mu)^{ \otimes p}  \right\rangle\neq 0$ then $a=m$ and hence $\mu=\lambda$.

To conclude the proof we first need to recall the notation introduced in Definition \ref{def:Amuro}. Let $\rho\in\mathcal{P}(p^{k})$ and $\mu\in\mathcal{P}(p^{k-1})$. Denote by $\mathcal{A}_\mu^\rho$ the set consisting of all the $\mu$-sequences of $\rho$
(this was previously denoted by $\mathcal{A}_{(\mu,\mu,\ldots,\mu)}^\rho$).
Let $\tau=(\mu_1,\mu_2,\ldots,\mu_p)\in A_\mu^\rho$. We let $d_\tau^\rho$ be the positive integer defined by 
$$d_\tau^\rho=\prod_{i=1}^{p-1}C^{\mu_{i+1}}_{\mu, \mu_i},$$
where $C^{\mu_{i+1}}_{\mu, \mu_i}$ denotes the usual Littlewood-Richardson coefficient. 
From Lemma \ref{lem:dlambda} we see that 
$$
\left\langle \chi^{\rho}\downarrow_{K}, (\chi^\mu)^{ \otimes p}\right\rangle = \sum_{\tau\in A_\mu^\rho}d_\tau^\rho.
$$
In the case where $\rho=h_j$ and $\mu=\lambda$, the above formula is particularly easy to use. 
As we noticed before each partition $\mu_{i+1}$ is a hook partition obtained from $\mu_i$ by adding $m+\varepsilon_i$ boxes to the first column and adding the remaining 
$p^{k-1}-(m+\varepsilon_i)$ boxes to the first row, for some $\varepsilon_i\in\{0,1\}$. Since $\ell(\lambda)=1+pm+x$ we deduce that for exactly $x$ values of $i\in\{1,\ldots, p-1\}$ we have that $\varepsilon_i=1$. This shows that $|A_\lambda^{h_j}|={p-1\choose x}$. 
Let now $\tau=(\mu_1,\ldots,\mu_p)\in A_\lambda^{h_j}$. Then $[\mu_{i+1}\smallsetminus\mu_i]$ is a skew partition of $p^{k-1}$ with one row of length $p^{k-1}-(m+\varepsilon_i)$, and one column of length $m+\varepsilon_i$, for some $\varepsilon_i\in\{0,1\}$. 
Hence $C^{\mu_{i+1}}_{\lambda, \mu_i}=1$ for all $i$, and therefore $d_\tau^{h_j}=1$. 
This shows that $\left\langle \chi^{h_j}\downarrow_{K}, (\chi^\lambda)^{ \otimes p}\right\rangle={p-1\choose x}$, as required. 
\end{proof}

Let $m\in\{0,1,\ldots,p^{k-1}-1\}$. Denote by $A_m$ the subset of $\mathrm{Irr}_{p'}(\fS_{p^{k-1}}\wr \fS_p)$ defined by 
$$A_m=\{ \chi((p^{k-1}-m,1^m),\mu)\, |\, \mu\in\mathcal{H}(p)\}.$$
Similarly let $B_m$ be the subset of $\mathrm{Irr}_{p'}(\fS_{p^k})$ defined by $$B_m=\{\chi^{(p^{k}-(pm+x),1^{pm+x})}\ |\ x\in\{0,1,\ldots,p-1\}\}.$$
Clearly $|A_m|=|B_m|=p$ for all $m\in\{0,1,\ldots,p^{k-1}-1\}$. Moreover we have that 
$$\mathrm{Irr}_{p'}(\fS_{p^{k-1}}\wr \fS_p)=\bigcup_{m=0}^{p^{k-1}-1}A_m,\ \ \text{and}\ \ \ \mathrm{Irr}_{p'}(\fS_{p^k})=\bigcup_{m=0}^{p^{k-1}-1}B_m.$$

\begin{proposition}\label{prop:1}
Let $m,\ell\in\{0,1,\ldots, p^{k-1}-1\}$. Let $\chi\in B_m$ and $\psi\in A_{\ell}$. If 
$\left\langle \psi, \chi\downarrow_{_{\fS_{p^{k-1}}\wr \fS_p}}\right\rangle\neq 0$ then $m=\ell$.
\end{proposition}
\begin{proof}
Suppose for a contradiction that $\psi\in A_\ell$ for some $\ell\in\{0,1,\ldots, p^{k-1}-1\}\smallsetminus\{m\}$.
Then $\psi=\chi((p^{k-1}-\ell,1^\ell),\mu)$, for some $\mu\in\mathcal{H}(p)$. By Lemma \ref{lem:1} we have that 
$$0=\left\langle \chi\downarrow_{_{(\fS_{p^{k-1}})^p}}, (\chi^{(p^{k-1}-\ell,1^\ell)})^{ \otimes p}\right\rangle\geq \left\langle \psi\downarrow_{_{(\fS_{p^{k-1}})^p}}, (\chi^{(p^{k-1}-\ell,1^\ell)})^{ \otimes p}\right\rangle=\chi^\mu(1).$$
This is a contradiction. 
\end{proof}

\begin{thm}\label{t:1}
Let $m\in\{0,1,\ldots, p^{k-1}-1\}$ and let $\chi\in B_m$. 
Then $\chi\downarrow_{_{\fS_{p^{k-1}}\wr \fS_p}}=\chi^\star+\Delta$, where $\chi^\star\in A_m$ and where $\Delta$ is a sum (possibly empty) of irreducible characters of degree divisible by $p$. Moreover the map $\chi\mapsto\chi^\star$ is a bijection between $B_m$ and $A_m$. 
\end{thm}

\begin{proof}
For $x\in\{0,1,\ldots, \frac{p-1}{2}\}$ let $A_m^x$ be the subset of $A_m$ defined by $$A_m^x=\{\chi((p^{k-1}-m,1^m),(p-y,1^y))\ \color{black} |\ y\in\{x, (p-1)-x\}\}.$$
Similarly let $B_m^x$ be the subset of $B_m$ defined by $$B_m^x=\{\chi^{(p^{k}-(mp+y),1^{mp+y})}\ |\ y\in\{x, (p-1)-x\}\}.$$
We will show by reverse induction on $x\in\{0,1,\ldots, \frac{p-1}{2}\}$ that the following statement holds. 

\noindent \textbf{Claim.} \textit{Let $\chi\in B_m^x$. Then $\chi\downarrow_{_{\fS_{p^{k-1}}\wr \fS_p}}=\chi^\star+\Delta$, where $\chi^\star\in A_m^x$ and where $\Delta$ is a sum (possibly empty) of irreducible characters of degree divisible by $p$. Moreover the map $\chi\mapsto\chi^\star$ is a bijection between $B_m^x$ and $A_m^x$.}

\medskip

\textbf{Base Step.} Let $\lambda=(p^{k-1}-m,1^m)\in\mathcal{H}(p^{k-1})$ and $\mu=(\frac{p+1}{2},1^{\frac{p-1}{2}})\in\mathcal{H}(p)$. Let $\psi:= \chi(\lambda, \mu)$ be the unique element of $A_m^{\frac{p-1}{2}}$.
To shorten the notation, let $K=\fS_{p^{k-1}}\wr \fS_p$.
Since $\psi(1)$ and $|\fS_{p^k}:K|$ are both coprime to $p$, there exist $\chi\in\mathrm{Irr}_{p'}(\fS_{p^k})$ such that 
$\chi$ is a constituent of $\psi\uparrow_{K}^{\fS_{p^k}}$. By Proposition \ref{prop:1} we deduce that $\chi\in B_m$.
Hence there exists $y\in\{0,1,\ldots, p-1\}$ such that $\chi=\chi^{(p^k-(pm+y),1^{pm+y})}$.
By Lemma \ref{lem:1} we obtain that 
$${p-1\choose y}=\left\langle \chi\downarrow_{_{(\fS_{p^{k-1}})^ p}}, (\chi^{\lambda})^{ \otimes p}\right\rangle\geq \left\langle \psi\downarrow_{_{(\fS_{p^{k-1}})^p}}, (\chi^{\lambda})^{ \otimes p}\right\rangle=\chi^\mu(1)={p-1\choose \frac{p-1}{2}}.$$
This clearly implies that $y=\frac{p-1}{2}$. Therefore we have that $B_m^{\frac{p-1}{2}}=\{\chi\}$ and that 
\begin{eqnarray*}
\left\langle \chi\downarrow_{K}, \phi\right\rangle=
\begin{cases}
1 &\text{ if } \phi=\psi\,,\\
0 &\text{ if } \phi\in\mathrm{Irr}_{p'}(K)\smallsetminus\{\psi\}\,.\\
\end{cases}
\end{eqnarray*}
This shows that $\chi^\star=\psi$, and proves the claim for $x=\frac{p-1}{2}$.

\smallskip

\textbf{Inductive Step.} Let $x\in\{0,1,\ldots, \frac{p-1}{2}-1\}$ and assume that the claim holds for all $x<y\leq\frac{p-1}{2}$. 
Clearly $|A_m^x|=|B_m^x|=2$. Let $\psi\in A_m^x=\{\psi, \psi_2\}$. Since $\psi(1)$ and $|\fS_{p^k}:K|$ are both coprime to $p$, there exist $\chi\in\mathrm{Irr}_{p'}(\fS_{p^k})$ such that 
$\chi$ is a constituent of $\psi\uparrow_{K}^{\fS_{p^k}}$. By Proposition \ref{prop:1} we deduce that $\chi\in B_m$.
By induction we deduce that $$\chi\in B_m\smallsetminus \big(\bigcup_{y=x+1}^{\frac{p-1}{2}}B_m^y\big)=\bigcup_{j=0}^{x}B_m^j.$$
Let $j\in\{0,1,\ldots, x\}$ be such that $\chi\in B_m^j$. By Lemma \ref{lem:1} we obtain that 
$${p-1\choose j}=\left\langle \chi\downarrow_{_{(\fS_{p^{k-1}})^p}}, (\chi^{\lambda})^{ \otimes p}\right\rangle\geq \left\langle \psi\downarrow_{_{(\fS_{p^{k-1}})^p}}, (\chi^{\lambda})^{ \otimes p}\right\rangle=\chi^\mu(1)={p-1\choose x}.$$
This necessarily implies that $j=x$. Therefore we have that $B_m^{x}=\{\chi, \chi_2\}$ and that 
\begin{eqnarray*}
\left\langle \chi\downarrow_{K}, \phi\right\rangle=
\begin{cases}
1 &\text{ if } \phi=\psi\,,\\
0 &\text{ if } \phi\in\mathrm{Irr}_{p'}(K)\smallsetminus\{\psi\}\,.\\
\end{cases}
\end{eqnarray*}
This shows that $\chi^\star=\psi$. 
We conclude by showing that $(\chi_2)^\star=\psi_2$.
Let $\chi'$ be a $p'$-degree constituent of $(\psi_2)\uparrow_K^{\fS_{p^k}}$. Using the inductive hypothesis and Lemma \ref{lem:1} exactly as above, we deduce that $\chi'\in B_m^x=\{\chi, \chi_2\}$. This immediately implies that $\chi'=\chi_2$. It follows by Lemma \ref{lem:1} that 
$\psi_2$ is the only constituent of degree coprime to $p$ of $\chi_2\downarrow_K$ and also that $\left\langle\chi_2\downarrow_K, \psi_2\right\rangle=1$. Hence
$(\chi_2)^\star=\psi_2$.
%

Therefore we have that $\chi\mapsto\chi^\star$ is a bijection between $B_m^x$ and $A_m^x$, as required. 
Taking the union of these sets over $x\in\{0,1,\ldots,\frac{p-1}{2}\}$ we obtain the statement of the theorem.
\end{proof}

As a straightforward consequence of Theorem \ref{t:1} we obtain the following result, which is a slightly stronger statement than Theorem B, as presented in the introduction. 

\begin{thm}\label{t:main correspondence p^k}
Let $K:=\fS_{p^{k-1}}\wr \fS_p$.
Let $\chi\in \mathrm{Irr}_{p'}(\fS_{p^k})$. The restriction $\chi\downarrow_K$ has a unique irreducible constituent $\chi^\star$ lying in $\mathrm{Irr}_{p'}(K)$, appearing with multiplicity 1. Moreover the map $\chi\mapsto\chi^\star$ is a bijection between $\mathrm{Irr}_{p'}(\fS_{p^k})$ and $\mathrm{Irr}_{p'}(K)$. 

More precisely if $\lambda=(p^k-(mp+x), 1^{mp+x}),$
for some $x\in\{0,1,\ldots, p-1\}$ and $\chi=\chi^\lambda$, then 
$\chi^*\in\{ \chi(\mu,\nu_1), \chi(\mu,\nu_2) \}$, where 
$$\mu=(p^{k-1}-m,1^m)\ \ , \ \nu_1=(p-x,1^x)\ \ \text{and}\ \ \nu_2=(x+1,1^{p-1-x}).$$
\end{thm}

In Section \ref{sec:4} below we will be able to give an ultimate complete description of the bijection given in Theorem \ref{t:main correspondence p^k} for the prime $p=3$. In particular in Corollary \ref{c:3precise} we will completely identify the character $\chi^*$.

\subsection{A local bijection}

We fix the prime $p=3$. 
Let $k\geq 1$, and suppose that $P_k$ is a Sylow $3$-subgroup
of the symmetric group $\fS_{3^k}$.  
Our aim in this section is to define a natural bijection of characters

$$
\Phi_{3^k}\colon \mathrm{Irr}_{3'}(\fS_{3^k})\longrightarrow \mathrm{Irr}_{3'}(\norm{\fS_{3^k}}{P_k})\, ,
$$ such that $\Phi_{3^k}(\chi)$ is an irreducible constituent of $\chi\downarrow_{\norm{\fS_{3^k}}{P_k}}$,
for each $\chi\in\mathrm{Irr}_{3'}(\fS_{3^k})$.

It is easy to check that $P_k$ induces on $\{1,\ldots,3^k\}$
a unique maximal block system consisting of three blocks of size $3^{k-1}$ each, 
and we denote by $H=\fS_{3^{k-1}}\wr \fS_3\leq \fS_{3^k}$ the stabilizer of the corresponding set partition.
(Recall that by Lemma \ref{complement} the complements $\fS_3$ of the base group
inducing the wretah product $H$ are all conjugate.) It is clear that $P_k\cap (\fS_{3^{k-1}})^3=(P_{k-1})^3$, 
and if $\langle z\rangle$
denotes the unique Sylow $3$-subgroup of $\fS_3$, then $P_k=(P_{k-1})^3\rtimes\langle z \rangle=:P_{k-1}\wr\langle z \rangle$. 
\color{black}
It is not difficult to show (see the proof of Lemma 4.2 in \cite{Olsson}) 
that $\norm {\fS_{3^k}}{P_k}\leq H$, 
so in particular $\norm {\fS_{3^k}}{P_k}=\norm {H}{P_k}$. 

In order to prove our main results in this section, we need first to recall some known facts on the 
structure of
the wreath product $H=\fS_{3^{k-1}}\wr \fS_3$, 
which are contained in Proposition 1.5 of \cite{Olsson} and its proof (see also Lemma \ref{norm} below). 
Let 
$$
M=\{ (x_1, x_2, x_3) \  |\ x_i\in\norm{\fS_{3^{k-1}}}{P_{k-1}}, \ x_i\equiv x_j\pmod {P_{k-1}} \mbox{\ for\ all\ } i,j \}\, .
$$ Then we have 
$\norm{H}{P_{k}}=M\fS_3= M\rtimes\fS_3\, .$ In particular, it is clear that
$$
\norm{H}{P_k}\leq \norm{\fS_{3^{k-1}}}{P_{k-1}}\wr \fS_3\, .
$$The elements in the derived subgroup of $P_k$ are precisely
$$
P_k'=\{ (x_1,x_2, x_3; 1)\, |\, x_i\in P_{k-1} \mbox{\ and\ } x_1x_2x_3\in P_{k-1}'\}\, ,
$$ so in particular  $P_k'\leq M$. Of course, note that $P_k'$ is normal in $\norm H P$. Furthermore, we have that
\begin{equation}\label{dprod}
\norm H {P_k}/{P_k}' \cong M/P_k' \times {\fS}_{3},
\end{equation} in a natural way. 

The following result contains a key step in the construction of the map $\Phi_{3^k}$ described 
above. 

\begin{proposition}\label{bij_wreath}
There exists a natural bijection between the sets $\mathrm{Irr}_{3'}(\norm{\fS_{3^{k-1}}}{P_{k-1}}\wr \fS_3)$ and $\mathrm{Irr}_{3'}(\norm{H}{P_k})$,
which is compatible with character restriction. 
\end{proposition}

\begin{proof}
Observe that the case $k=1$ is trivial, because $\fS_3$ has a normal Sylow $3$-subgroup, so assume that 
$k\geq 2$.  We divide the proof of the result into several steps.


\begin{Step}\label{map}
{ There exists a natural map $\lambda^{\otimes 3}\mapsto \psi_\lambda$
from the set of 
$3'$-degree $\fS_3$-invariant irreducible characters of $\norm{\fS_{3^{k-1}}}{P_{k-1}}^3$, into
the set of irreducible characters of $3'$-degree of $M$ whose kernel contains $P_{k}'$.}
\end{Step}

\begin{proof}[Proof of Step \ref{map}. ]
Note that any $3'$-degree irreducible character of $\norm{\fS_{3^{k-1}}}{P_{k-1}}^3$
that is $\fS_3$-invariant is of the form $\lambda^{\otimes 3}$, where $\lambda\in\mathrm{Irr}_{3'}(\norm{\fS_{3^{k-1}}}{P_{k-1}})$.
So assume $\lambda\in\mathrm{Irr}_{3'}(\norm{\fS_{3^{k-1}}}{P_{k-1}})$, and
let $\alpha\in\irr{P_{k-1}}$ be an irreducible constituent of $\lambda_{P_{k-1}}$. 
In particular, note that $\alpha$ is linear. Of course, by Clifford's Theorem 6.2 of \cite{isaacs}, $\alpha$ is unique up to 
$\norm{\fS_{3^{k-1}}}{P_{k-1}}$-conjugacy. Let  $\eta=\alpha^{\otimes 3} \in\irr{(P_{k-1})^{3}}$. 
If $I$ denotes
the stabilizer of $\alpha$ in $\norm{\fS_{3^{n-1}}}{P_{k-1}}$, 
then $J=I^3$ is the stabilizer of $\eta$
in $\norm{\fS_{3^{k-1}}}{P_{k-1}}^3$.
By Clifford's correspondence, let $\hat \alpha\in\irr {I\, |\, \lambda}$ be uniquely determined by 
$\hat\alpha\uparrow^{\norm {\fS_{3^{k-1}}}{P_{k-1}}}=\lambda$. 
Then $\varphi=\hat\alpha^{\otimes 3}\in\irr{J}$ lies over $\eta$ and induces to
$\lambda^{\otimes 3}$, again by Clifford's correspondence. Also, note that $\varphi$ is an extension of $\eta$
by Gallagher's theorem  \cite[Corollary 6.17]{isaacs}, 
because $\alpha$ extends to $I$ by Theorem \ref{canonical_ext}, and $\norm{\fS_{3^{k-1}}}{P_{k-1}}/P_{k-1}$ is abelian (see Lemma 3.3 of \cite{MO}, for instance). Let $L=M\cap J$, so $\varphi\downarrow_L\in\irr{L\, |\, \eta}$. 
Now, by Clifford's correspondence, we have that $\psi=(\varphi\downarrow_L){\big\uparrow}^M\in\irrq{M}{3}$.

It is not difficult to check, by using the uniqueness in Clifford's correspondence,
that the definition of the character $\psi$ in the previous paragraph depends only on $\lambda$, 
which is the same to say that it is independent of the choice
of $\alpha$. 
Thus, we may write $\psi=\psi_\lambda$.

In order to finish the proof of this step, we need to show that $P'\leq \ker{\psi_\lambda}$, 
for any $\lambda\in\irrq{\norm{\fS_{3^{k-1}}}{P_{k-1}}}{3}$. 
First, note that
$\psi_\lambda$ is 
${\fS}_3$-invariant. In particular, since $\fS_3$ has a cyclic
Sylow $3$-subgroup
we deduce that $\psi_\lambda$ extends to $MP_{k}$.
Now, by elementary character theory
$P'$ is contained in the kernel of any extension of  $\psi_\lambda$ to $MP$, and this completes the proof of this step.
\end{proof}

\begin{Step}\label{map2}
{The map defined in Step \ref{map} is bijective.}
\end{Step}

\begin{proof}[Proof of Step \ref{map2}. ]
First we show that the map in Step \ref{map} is surjective. Let $\psi\in\irrq{M}{3}$
with $P'\leq \ker \psi$. It is clear from (\ref{dprod}) that
$\psi$ is $\fS_3$-invariant, and in particular it is fixed by the Sylow $3$-subgroup
$\langle z \rangle$ of ${\fS}_3$. Since $\psi(1)$ is coprime to $3$, 
an easy argument on character degrees yields that $\psi$ lies over a 
$3'$-degree irreducible character of $(P_{k-1})^3$ which is $\langle z \rangle$-invariant, that is $\psi$ lies
over 
a character of the form $\eta=\alpha^{\otimes 3}\in\irr{(P_{k-1})^3}$, 
where $\alpha\in\irr{P_{k-1}}$ is linear. 
Let $I$ be the stabilizer of $\alpha$ in $\norm{\fS_{3^{k-1}}}{P_{k-1}}$, and write $J=I^3$ and $L=M\cap J$. 
By Clifford's correspondence, let $\gamma\in\irr{L\, |\, \eta}$ such that $\gamma\uparrow^M=\psi$. 
Since $(P_{k-1})^3$ has index coprime to $3$ in $J$, by Gallagher's theorem \cite[Corollary 6.17]{isaacs}
we have that $\gamma=\tau\cdot\delta$, where $\tau\in\irr{L}$ is the canonical extension
of $\eta$ (in the sense of Theorem \ref{canonical_ext}) and $\delta\in\irr{L/K}$.  
Now let $\hat\tau\in\irr{J}$ be the canonical extension of $\eta$
to $J$, so $\hat \tau$ restricts to $\tau$ by the properties of the canonical extension. 
Note that $L$ is $\langle z \rangle$-invariant. 
Also, note that $\delta$ is $\langle  z \rangle$-fixed.
Since $J/K$ is a abelian, it is clear that 
$\delta$ extends to $J$, and since 
$J/K$ is a $2$-group (see Lemma 3.3 of \cite{MO}),  we deduce by coprime action that there 
exists a $\langle z \rangle$-invariant extension $\hat\delta$ of $\delta$ to $J$.
Then
$\varphi=\hat \tau\cdot\hat \delta$ is a $\langle z \rangle$-invariant extension of $\eta$
to $J$. It is clear that $\varphi=\hat\alpha^{\otimes 3}$, for some
extension $\hat\alpha$ of $\alpha$ to $I$. Let $\lambda={\hat\alpha}\uparrow^{\norm{\fS_{3^{k-1}}}{P_{k-1}}}$, which
of course is an irreducible character of degree coprime to $3$. By construction we have that $\psi=\psi_{\lambda}$, so the map
in the previous step is surjective, as wanted. 
Similar considerations lead to the conclusion that the map in Step \ref{map} is also injective, 
which completes the proof of this step.

\end{proof}

\begin{Step}\label{bijection}
{If $\lambda^{\otimes 3}\mapsto \psi_\lambda$ under the map defined in Step \ref{map}, then
there exists a natural bijection

$$
f_\lambda\colon \irrq{\norm{\fS_{3^{k-1}}}{P_{k-1}}\wr \fS_3 \, |\, \lambda^{\otimes 3}}{3}\longrightarrow \irrq{\norm{H}{P}\, |\, \psi_\lambda}{3}
$$ which is compatible with character restriction. 
}\end{Step}

\begin{proof}[Proof of Step \ref{bijection}. ]

As before, write $\psi=\psi_\lambda$, and
let $\hat\psi=\varphi\uparrow^{JM}$.
Recall that $\psi=\hat\psi\downarrow_M$, and that since $\varphi$ is invariant under $\fS_3$, it follows that 
$\hat\psi$ is also $\fS_3$-invariant. 
Now, by Proposition \ref{isa_lemma}(\ref{restriction}) restriction defines a bijection from $\irr{JM\rtimes{\fS_{3}} \, |\, \hat\psi}$
into $\irr{\norm H P\, |\, \psi}$. Observe that
$\hat\psi$ induces irreducibly to $\lambda^{\otimes 3}$ by Clifford's correspondence, and thus
by Proposition \ref{isa_lemma}(\ref{induction}) induction defines a bijection from
$\irr{JM\rtimes{\fS_{3}} \, |\, \hat\psi}$ 
into $\irr{\norm{\fS_{3^{k-1}}}{P_{k-1}}\wr {\fS}_3\, |\, \lambda^{\otimes 3}}$. 
By composing the inverse of the latter bijection with the former one, and then restricting the resulting map to the set of irreducible characters of degree coprime to $3$, one obtains a correspondence 

$$
f_\lambda\colon \irrq{\norm{\fS_{3^{k-1}}}{P_{k-1}}\wr \fS_3 \, |\, \lambda^{\otimes 3}}{3}\longrightarrow \irrq{\norm{H}{P}\, |\, \psi_\lambda}{3}
$$ as desired. 

\end{proof}

\begin{Step}\label{final}
{ The union $f$ of the maps $f_\lambda$ defined in Step $2$, with $\lambda$ running over 
$\irrq{\norm{\fS_{3^{k-1}}}{P_{k-1}}}{3}$, is a natural bijection

$$f\colon\mathrm{Irr}_{3'}(\norm{\fS_{3^{k-1}}}{P_{k-1}}\wr \fS_3)\longrightarrow \mathrm{Irr}_{3'}(\norm{H}{P})$$
which is compatible with character restriction. }
\end{Step}

\begin{proof}[Proof of Step \ref{final}. ]

Let $\xi\in \mathrm{Irr}_{3'}(\norm{\fS_{3^{k-1}}}{P_{k-1}}\wr \fS_3)$. It is easy to see that
$\xi$ lies over a unique $3'$-degree irreducible character of $\norm{\fS_{3^{k-1}}}{P_{k-1}}^3$, which
necessarily is of the form $\lambda^{\otimes 3}$ for a unique $\lambda\in\irrq{\norm{\fS_{3^{k-1}}}{P_{k-1}}}{3}$. 
Thus $f(\xi)=f_\lambda(\xi)$, and  the union $f$ of the maps defined in Step 1 is well-defined.

Suppose now that $\mu\in \mathrm{Irr}_{3'}(\norm{H}{P})$, so in particular we have that $P'\leq \ker \mu$. Then 
it follows from (\ref{dprod}) that $\mu$ lies over a uniquely determined character $\psi\in\irrq{M}{3}$ 
with $P'\leq \ker \psi$, which necessarily is $\fS_3$-invariant. Thus, by Step \ref{bijection} in order to show that $f$ is surjective,
it is enough to see that for any $\psi\in\irrq{M}{3}$ 
with $P'\leq \ker \psi$, there exists $\lambda\in\irrq{\norm{\fS_{3^{k-1}}}{P_{k-1}}}{3}$ such that $\psi=\psi_\lambda$. 
Of course, this is contained in Step \ref{map2}. 

Finally, it is easy to deduce from Step \ref{map2}
that the map $f$ is injective, 

\end{proof}
We have completed the proof of Proposition \ref{bij_wreath}.
\end{proof}

\begin{thm}\label{t:2main correspondence p^k}
There exists a canonical bijection $\Phi_{3^k}$ between
$\mathrm{Irr}_{3'}(\fS_{3^k})$ and $\mathrm{Irr}_{3'}(\norm{\fS_{3^k}}{P_k})$. Moreover, 
$\Phi_{3^k}(\chi)$ is an irreducible constituent of the restricted character $\chi\downarrow_{\norm{\fS_{3^k}}{P_k}}$, for all $\chi\in\mathrm{Irr}_{3'}(\fS_{3^k})$.
\end{thm}

\begin{proof}
We work by induction on $k\geq 1$. If $k=1$ then the result is clear, 
since $\fS_3$ has a normal Sylow $3$-subgroup, so assume that $k\geq 2$. 

Suppose that the result holds in $\fS_{3^{k-1}}$, and let $P_k\in\syl{3}{\fS_{3^k}}$. 
Let $H=\fS_{3^{k-1}}\wr \fS_3\leq \fS_{3^k}$ be determined by $P_k$ as in the beginning of this section. 
By Theorem \ref{t:main correspondence p^k}, we need to define a canonical bijection
$$
\irrq{H}{3}\longrightarrow \irrq{\norm{\fS_{3^{k}}}{P_{k}}}{3}
$$ which is compatible with character restriction. As before, write $P_{k}\cap (\fS_{3^{k-1}})^3=(P_{k-1})^3$. 
Then, observe that by Proposition \ref{bij_wreath}, it suffices to define 
a natural correspondence 
$$
\irrq{H}{3}\longrightarrow \mathrm{Irr}_{3'}(\norm{\fS_{3^{k-1}}}{P_{k-1}}\wr \fS_3)\, 
$$ which respects restriction of characters. 

Let $\psi=\theta^{\otimes 3}\in\irrq{(\fS_{3^{k-1}})^3}{3}$, where $\theta\in\irrq{\fS_{3^{k-1}}}{3}$.  
By the inductive hypothesis, let $\lambda=\Phi_{3^{k-1}}(\theta)$, and write 
$\nu=\lambda^{\otimes 3}\in\irrq{\norm{\fS_{3^{k-1}}}{P_{k-1}}^3}{3}$. 
In particular, note that $\nu$ is a constituent of $\psi\downarrow_{\norm{\fS_{3^{k-1}}}{P_{k-1}}^3}$. 
Let $\tilde {\psi}$ (respectively $\tilde {\nu}$) be the extension of $\psi$ (resp. $\nu$) 
to the wreath product $H$ (resp. $\norm{\fS_{3^{k-1}}}{P_{k-1}}\wr \fS_3$)
described in Lemma \ref{wr-ext}. 
Then 
it is easy to see that $\tilde\nu$ is
a constituent of the restriction of $\tilde\psi$ to ${\norm{\fS_{3^{k-1}}}{P_{k-1}}\wr \fS_3}$. 
Thus, the map $\tilde\psi\cdot\delta\mapsto \tilde\nu\cdot\delta$, where $\delta\in\irr{\fS_3}$, is a natural bijection
$$
\irrq{H\, |\, \psi}{3}\longrightarrow \mathrm{Irr}_{3'}(\norm{\fS_{3^{k-1}}}{P_{k-1}}\wr \fS_3\, |\, \nu)\, 
$$ which respects restriction of characters. Now, it is easy to see that the union of these maps, where $\psi$
runs in the set of $\fS_3$-invariant characters in $\irrq{(\fS_{3^{k-1}})^3}{3}$
and $\nu\in\irrq{\norm{\fS_{3^{k-1}}}{P_{k-1}}^3}{3}$ is defined accordingly as above,
is a canonical bijection as desired. 
\end{proof}

\section{A canonical McKay bijection for $\fS_{n}$}\label{sec:4}

In order to deal with the general case in symmetric groups and construct a canonical bijection between $\mathrm{Irr}_{3'}(\fS_n)$ and $\mathrm{Irr}_{3'}(\NB_{\fS_n}(P_n))$, we need to analyse first the case where $n=2\cdot 3^k$, for some $k\in\mathbb{N}$.

\subsection{The case $n=2\cdot 3^k$}   \label{sec:intro}
We start with some observations on the structure of the two sets $\mathrm{Irr}_{3'}(\fS_{2\cdot 3^k})$ and $\mathrm{Irr}_{3'}(\NB_{\fS_{2\cdot 3^k}}(P_{2\cdot 3^k}))$.
We first introduce the following useful definition.

\begin{definition}\label{hookgen}
Let $h_1, h_2\in\mathcal{H}(3^k)$ be two distinct hook partitions, $h_1\neq h_2$. We say that a partition $\lambda\in\mathcal{P}(2\cdot 3^k)$ is \textit{hook-generated} by the pair $\{h_1,h_2\}$ if $\lambda$ has two removable $3^k$-hooks $k_1, k_2$ and we have that $\lambda- k_i=h_i$ for $i\in\{1,2\}$. 
\end{definition}

\begin{remark}\label{remwelldef}
The concept introduced in \ref{hookgen} is well defined since for every pair $\{h_1,h_2\}$ of distinct $3^k$-hooks there exists a unique partition $\lambda\in\mathcal{P}(2\cdot 3^k)$ such that $\lambda$ is hook-generated by $\{h_1,h_2\}$. This can be seen by looking at the $3$-core towers of the partitions involved. From repeated applications of \cite[Theorem 3.3]{OlssonBook} we deduce that for any $i\in\{1,2\}$ there exists a unique $x_i\in\{1,\ldots, 3^k\}$ such that 
$T_j(h_i)=(\emptyset,\ldots, \emptyset)$ for all $j<k$ and $T_k(h_i)=((h_i)_1^k,\ldots,(h_i)_{3^k}^k)$, where $(h_i)_{x_i}^k=(1)$ and $(h_i)_{s}^k=\emptyset$, for all $s\in\{1,\ldots, 3^k\}\smallsetminus\{x_i\}$. Since $h_1\neq h_2$ it follows that $x_1\neq x_2$. 
If $\lambda\in\mathcal{P}(2\cdot 3^k)$ is hook generated by $\{h_1,h_2\}$, then it follows that $T_j(\lambda)=(\emptyset,\ldots, \emptyset)$ for all $j<k$ and $T_k(\lambda)=(\lambda_1^k,\ldots,\lambda_{3^k}^k)$, where $\lambda_{x_i}^k=(1)$ for all $i\in\{1,2\}$ and $\lambda_{s}^k=\emptyset$, for all $s\in\{1,\ldots, 3^k\}\smallsetminus\{x_1,x_2\}$. Since every partition is uniquely determined by its $3$-core tower, we have that $\lambda$ is unique. 
\end{remark}

We denote by $\mathcal{HG}(2\cdot 3^k)$ the subset of $\mathcal{P}(2\cdot 3^k)$ consisting of all partitions that are hook-generated by some pair of distinct $3^k$-hooks. 
Definition \ref{hookgen} and Remark \ref{remwelldef} show that a partition $\lambda$ lies in $\mathcal{HG}(2\cdot 3^k)$ if and only if its $3$-core tower $T(\lambda)$ has exactly two non-empty $3$-cores lying in the $k$-th layer. Both these $3$-cores are equal to $(1)$.

\medskip

Let $A=\{\chi^h\ |\ h\in\mathcal{H}(2\cdot 3^k)\}$ and let $B=\{\chi^\lambda\ |\ \lambda\in\mathcal{HG}(2\cdot 3^k)\}$.
From Theorem \ref{Mac} we have that $\mathrm{Irr}_{3'}(\fS_{2\cdot 3^k})=A\cup B$.

\begin{definition}
Let $n,m\in\mathbb{N}$. Given $\omega\in\{(n), (1^n)\}$ and $\lambda=(m-x,1^x)\in\mathcal{H}(m)$ we let 
\begin{eqnarray}
\lambda\bullet \omega =
\begin{cases}
(m-x+n,1^x) &\text{ if } \omega=(n)\,,\\
(m-x,1^{x+n}) &\text{ if } \omega=(1^n)\,.\\
\end{cases}
\end{eqnarray}
Moreover, we write $\varepsilon(\lambda\bullet\omega)=\delta_{\lambda\lambda'}+\delta_{(1^{n})\omega}$, where $\delta_{xy}=1$ if $x=y$ and $\delta_{xy}=0$ if $x\neq y$.
\end{definition}
\noindent It is easy to observe that $A=\{h\bullet\omega\  |\  h\in\mathcal{H}(3^k), \omega\in\{(3^k), (1^{3^k})\}\ \}$. 

We next describe the set of $3'$-degree irreducible characters
of $\fS_{3^k}\wr \fS_2\leq \fS_{2\cdot 3^{k}}$.
As in Lemma \ref{complement}, it is easy to see that all complements of the base group $(\fS_{3^k})^2$
in the wreath product $\fS_{3^k}\wr \fS_2$ are conjugate. Thus, for each $\chi=\chi^h\in\irrq{\fS_{3^k}}{3}$
with $h\in\mathcal H(3^k)$, Lemma \ref{wr-ext} provides a uniquley defined extension $\tilde\chi$ of
the character 
$\chi^{\otimes 2}\in\irrq{(\fS_{3^k})^2}{3}$
to $\fS_{3^k}\wr \fS_2$. If $\mu\in\mathcal P(2)$, we denote
$$
\chi(h, \mu):=\tilde\chi\cdot\chi^\mu\in\irr{\fS_{3^k}\wr \fS_2}\, ,
$$ where of course $\chi^\mu\in\irr{\fS_2}$ is identified with its inflation to $\fS_{3^k}\wr\fS_2$. 
\color{black}
Now it is easy to see that $\mathrm{Irr}_{3'}(\fS_{3^k}\wr \fS_2)=C\cup D$, where 
$$C=\{ \chi(h,\mu)\color{black} \ |\ h\in\mathcal{H}(3^k)\ \text{and}\ \mu\in\mathcal{P}(2)\}$$ and 
$$D=\{(\chi^{h_1}\otimes\color{black} \chi^{h_2})\uparrow_{(\fS_{3^k})^2}^{\fS_{3^k}\wr \fS_2}\ |\ h_1,h_2\in\mathcal{H}(3^k)\ \text{and}\ h_1\neq h_2\}.$$

We now let $\Psi_1: A\longrightarrow C$ be the map defined by 
\begin{eqnarray}
\Psi_1(\chi^{h\bullet\omega})=
\begin{cases}
\chi(h,(2)) &\text{ if } \varepsilon(h\bullet\omega) \text{ is even},\\
\chi(h,(1^2)) &\text{ if } \varepsilon(h\bullet\omega) \text{ is odd}.\\
\end{cases}
\end{eqnarray}

On the other hand we let $\Psi_2: B\longrightarrow D$ be the map defined by 
$$\Psi_2(\chi^\lambda)=(\chi^{h_1}\otimes\color{black} \chi^{h_2})\uparrow_{(\fS_{3^k})^2}^{\fS_{3^k}\wr \fS_2},$$
where $\{h_1,h_2\}$ hook-generates $\lambda$.

\begin{thm}\label{t:2Psi}
The map $\Psi:\mathrm{Irr}_{3'}(\fS_{2\cdot 3^k})\longrightarrow\mathrm{Irr}_{3'}(\fS_{3^k}\wr \fS_2)$, defined by $\Psi(a)=\Psi_1(a)$ for all $a\in A$, and by  $\Psi(b)=\Psi_2(b)$ for all $b\in B$, is a canonical bijection. 
\end{thm}
\begin{proof}
The map $\Psi_1$ is clearly a bijection between $A$ and $C$. The map $\Psi_2$ is well defined, since any partition $\lambda\in\mathcal{HG}(2\cdot 3^k)$ is hook-generated by a unique pair of distinct $3^k$-hooks. This can be proved with an argument similar to the one used in remark \ref{remwelldef}. The map is clearly surjective, and it is injective by Remark \ref{remwelldef}.
Since both $\Psi_1$ and $\Psi_2$ are choice-free bijections, we have that so is $\Psi$. 

Finally we observe that $\Psi$ is equivariant with respect to group automorphisms. 
 This is straightforward for inner automorphisms of $\fS_{2\cdot 3^k}$, and then it can be checked explicitly 
for the general case
after computing $\Psi$ for $\fS_6$ (when $k=1$), see Example \ref{s6} below. 
\end{proof}

We remark that a  similar map was constructed in a different context by Evseev in \cite{Evseev}.
\begin{example}\label{s6}
Let $n=6$. The partition labelling the irreducible $3'$-characters of $\fS_6$ are $\mathcal{H}(6)\cup \{(3,3); (3,2,1); (2,2,2)\}$. 
In the equations below we completely describe 
the 
%
bijection $\Psi$ obtained in Theorem \ref{t:2Psi} in the case of $\fS_6$.
\begin{eqnarray*}
\chi^{(6)} \mapsto  \chi((3),(2)) &,&\  
\chi^{(5,1)} \mapsto  \chi((2,1),(1^2)),\ \ \ \ \ \ \ \
\chi^{(4,1^2)} \mapsto  \chi((1^3),(2))\\ 
\chi^{(3,1^3)} \mapsto  \chi((3),(1^2))&,&\ 
\chi^{(2,1^4)} \mapsto  \chi((2,1),(2)),\ \ \ \ \ \ \ \ 
\chi^{(1^6)} \mapsto  \chi((1^3),(1,1))\\
\chi^{(3,3)} \mapsto  (\chi^{(3)}\otimes\chi^{(2,1)})\uparrow^{\fS_3\wr\fS_2} &,&\ 
\chi^{(3,2,1)} \mapsto  (\chi^{(3)}\otimes\chi^{(1^3)})\uparrow^{\fS_3\wr\fS_2},\
\chi^{(2,2,2)} \mapsto  (\chi^{(1^3)}\otimes\chi^{(2,1)})\uparrow^{\fS_3\wr\fS_2}\hspace{-2mm}.
\end{eqnarray*}
\end{example}

\medskip

Let $P_r$ denote a Sylow $3$-subgroup of $\fS_r$. It is clear that
$P_{2\cdot 3^k}$ partitions the set $\{1,\ldots,2\cdot~3^k\}$ into
two orbits of size $3^k$, and that $P_{2\cdot 3^k}=(P_{3^k})^2\leq (\fS_{3^k})^2$, where $(\fS_{3^k})^2$ is
the Young subgroup of $\fS_{2\cdot 3^k}$ associated to that set partition. Now, it is not difficult to see that

$$
\norm{\fS_{2\cdot 3^k}}{P_{2\cdot3^k}}=\norm{\fS_{3^k}}{P_{3^k}}\wr \fS_2\, ,
$$and the complements of the base group in this wreath product are all conjugate.
Write $N_r:=\NB_{\fS_{r}}(P_{r})$ for $r\in\mathbb N$, so 
$N_{2\cdot 3^k}=N_{3^k}\wr \fS_2\leq \fS_{3^k}\wr \fS_2$. If $\phi\in\irr{N_{3^k}}$, we denote by 
$\tilde\phi\in\irr{N_{2\cdot 3^k}}$
the extension of $\phi^{\otimes 2}\in\irr{(N_{3^k})^2}$ prescribed by Lemma \ref{wr-ext}. Then
we have that $\mathrm{Irr}_{3'}(N_{2\cdot 3^k})$ equals to
$$\{\tilde\phi\cdot\chi^\mu\ |\ \phi\in\mathrm{Irr}_{3'}(N_{3^k}),\ \mu\in\mathcal{P}(2)\}\cup\{(\phi\otimes\psi)\uparrow^{N_{2\cdot 3^k}}_{(N_{3^k})^2}\ |\ \phi,\psi\in\mathrm{Irr}_{3'}(N_{3^k}),\ \phi\neq\psi\}.$$

\color{black}
Now, Theorem \ref{t:2main correspondence p^k} together with the structure of $\mathrm{Irr}_{3'}(\fS_{3^k}\wr \fS_2)=C\cup D$ as discussed at the beginning of Section \ref{sec:intro}, shows that the following holds. 

\begin{proposition} \label{prop:2Theta}
The map $\Theta:\mathrm{Irr}_{3'}(\fS_{3^k}\wr \fS_2)\longrightarrow \mathrm{Irr}_{3'}(N_{2\cdot 3^k})$, defined by 
$$
 \Theta\big(\chi(\lambda,\mu)\big)=\tilde\phi_\lambda\cdot \chi^\mu
$$ 
for all $\chi(\lambda,\mu)\in C$, where $\phi_\lambda=\Phi_{3^k}(\chi^\lambda)\in\irrq{N_{3^k}}{3}$; and by 
$$\Theta\big((\chi^\lambda\otimes\color{black} \chi^\nu)\uparrow^{\fS_{3^k}\wr \fS_2}_{(\fS_{3^k})^2}\big)=(\Phi_{3^k}(\chi^\lambda)\otimes\color{black} \Phi_{3^k}(\chi^\nu))\uparrow^{N_{2\cdot 3^k}}_{(N_{3^k})^2}$$ for all $(\chi^\lambda\otimes\color{black} \chi^\nu)\uparrow^{\fS_{3^k}\wr \fS_2}_{(\fS_{3^k})^2}\in D$, 
is a canonical bijection. 
Moreover, $\Theta(\chi)$ is an irreducible constituent of $\chi\downarrow_{N_{2\cdot 3^k}}$, for all $\chi\in\mathrm{Irr}_{3'}(\fS_{3^k}\wr \fS_2)$.
\end{proposition}

\begin{thm}\label{t:main23^k}
The map $\Phi_{2\cdot 3^k}$ obtained as the composition of the maps $\Psi$ and $\Theta$ is a canonical bijection between $\mathrm{Irr}_{3'}(\fS_{2\cdot 3^k})$ and $\mathrm{Irr}_{3'}(N_{2\cdot 3^k})$.
\end{thm}

\medskip
In the last part of this section
we focus on giving a more precise description 
of 
the canonical bijection constructed in Theorem \ref{t:main correspondence p^k} for the prime $p=3$.

\begin{proposition}\label{p:hooks23^k}
Let $n=2\cdot 3^k$, let $h=(n-\ell,1^\ell)\in\mathcal{H}(n)$ and let $K=\fS_{3^k}\wr \fS_2$.
The restriction $(\chi^h)\downarrow_{K}$ has a unique irreducible constituent $\phi^h$ lying in $C$. 

More precisely if $m\in\mathbb{N}$ is such that $\ell=2m+x$ for some $x\in\{0,1\}$, then $\phi^h=\chi(\lambda,\mu)$ where $\lambda=(3^k-m,1^m)$ and $\mu=(2)$ if $m+x$ is even, $\mu=(1,1)$ if $m+x$ is odd. 
\end{proposition}

\begin{proof}
Let $\mu=(3^k-y,1^y)\in\mathcal{H}(3^k)$ and let $\nu\in\mathcal{H}(2\cdot 3^k)$ be such that $\chi^\nu$ is an irreducible constituent of 
$(\chi^\mu\otimes \chi^\mu)\uparrow_{(\fS_{3^k})^2}^{\fS_{2\cdot 3^k}}$. An easy application of the Littlewood-Richardson rule shows that 
$2y+1\leq\ell(\nu) \leq 2(y+1)$. We deduce that 
\begin{eqnarray}\label{eq3}
\left\langle \chi^h\downarrow_{(\fS_{3^k})^2} , \chi^\mu\otimes\color{black} \chi^\mu\right\rangle=
\begin{cases}
1 &\text{ if } \mu=\lambda\,,\\
0 &\text{ if } \mu\in\mathcal{H}(3^{k})\smallsetminus\{\lambda\}\,.\\
\end{cases}
\end{eqnarray}
This implies that there exists $\nu\in\mathcal{P}(2)$ such that $\phi^h:=\chi(\lambda, \nu)$ is the unique irreducible constituent of $(\chi^h)\downarrow_{K}$ lying in $C$. 
Equivalently, $(\chi^h)\downarrow_K=\phi^h + \Delta$, where $\Delta$ is a sum (possibly empty) of irreducible characters of $K$ of the form $(\chi^\tau\otimes\color{black}\chi^\sigma)\uparrow_{(\fS_{3^k})^2}^K$, for some $\sigma, \tau\in\mathcal{H}(3^k)$. In particular $\Delta(\gamma)=0$, where $\gamma\in K\leq \fS_{2\cdot 3^k}$ is the element of cycle type $(2^{3^k})$ such that $K=\big(\fS_{3^k}\times \fS_{3^k}^\gamma\big)\rtimes\langle\gamma\rangle$.
Using the Murnaghan-Nakayama formula we deduce that 
$${3^k-1\choose m}(-1)^{m+x}=\chi^h(\gamma)=\phi^h(\gamma)=\chi^\lambda(1)\chi^\nu((12))={3^k-1\choose m}\chi^\nu((12)).$$
We conclude that $\nu=(2)$ if $m+x$ is even, $\nu=(1,1)$ if $m+x$ is odd. 
The proof is concluded. 
\end{proof}

\begin{lemma}\label{l:231}
Let $k\in\mathbb{N}_{\geq 2}$ and let $m\in\{0,1,\ldots, 3^{k-1}-1\}$. 
Let $\lambda=(3^k-3m,1^{3m})\in\mathcal{H}(3^k)$, $\alpha=(2\cdot 3^{k-1}-2m,1^{2m})\in\mathcal{H}(2\cdot 3^{k-1})$
and $\mu=(3^{k-1}-m,1^m)\in\mathcal{H}(3^{k-1})$. 
Then $\psi=\chi^\alpha\otimes\color{black}\chi^\mu$ is the unique irreducible character of $\fS_{2\cdot 3^k}\times\fS_{3^k}$ such that 
$$\left\langle \chi^{\lambda}\downarrow_{\fS_{2\cdot 3^{k-1}}\times\fS_{3^{k-1}}},\psi \right\rangle\neq 0 \neq \left\langle \psi\downarrow_{(\fS_{3^{k-1}})^3}, \chi^\mu\otimes\color{black}\chi^\mu\otimes\color{black}\chi^\mu\right\rangle.$$
Moreover, $\left\langle\chi^{\lambda}\downarrow_{\fS_{2\cdot 3^{k-1}}\times\fS_{3^{k-1}}},\psi \right\rangle=1$.
\end{lemma}

\begin{proof}
To ease the notation we let $A=\fS_{2\cdot 3^{k-1}}\times\fS_{3^{k-1}}$ and $B=(\fS_{3^{k-1}})^3$. 
Let $\phi=\chi^{\nu}\otimes\color{black}\chi^{\mu}$ be an irreducible constituent of $\chi^\lambda\downarrow_A$ and of $(\chi^\mu\otimes\color{black}\chi^\mu\otimes\color{black}\chi^\mu)\uparrow^{A}$.
We want to show that $\nu=\alpha$.
Since $\phi$ lies below $\chi^\lambda$ we deduce that $\nu\in\mathcal{H}(2\cdot 3^k)$.
Moreover, by equation (\ref{eq3}) in the proof of Proposition \ref{p:hooks23^k} we obtain that $\nu\in\{(2\cdot 3^{k-1}-(2m+1),1^{2m+1}), \alpha\}$. 
Suppose that $\nu=(2\cdot 3^{k-1}-(2m+1),1^{2m+1})$ and let $\rho\in\mathcal{H}(3^k)$ be such that $\chi^\rho$ is an irreducible constituent of $\phi\uparrow^{\fS_{3^k}}$. Then we would have that 
$\ell(\rho)\geq 3m+2 > 3m+1=\ell(\lambda)$. This is a contraddiction. 
Hence $\nu=\alpha$ and $\phi=\psi$. 

The second statement follows easily because 
$\left\langle \chi^\lambda\downarrow_{B}, \chi^\mu\otimes\color{black}\chi^\mu\otimes\color{black}\chi^\mu\right\rangle=1$ by Lemma \ref{lem:1}.
\end{proof}

As a corollary we obtain the desired description of the canonical bijection constructed in Theorem \ref{t:main correspondence p^k} for the prime $p=3$. 

\begin{corollary}\label{c:3precise}
Let $k\in\mathbb{N}_{>0}$ and let $\lambda=(3^k-(3m+x),1^{3m+x})\in\mathcal{H}(3^k)$ and $\mu=(3^{k-1}-m,1^m)\in\mathcal{H}(3^{k-1})$, for some $x\in\{0,1,2\}$ and $m\in\{0,1,\ldots,3^{k-1}-1\}$. 
Write $\theta\in\irr{\fS_{3^{k-1}}\wr\fS_3}$ for the extension of
$(\chi^\mu)^{\otimes 3}\in\irr{(\fS_{3^{k-1}})^3}$ described in Lemma \ref{wr-ext}. \color{black}
Then the unique irreducible constituent of degree coprime to $3$ of 
$\chi^\lambda\downarrow_{\fS_{3^{k-1}}\wr \fS_3}$ is 

\begin{eqnarray*}
\chi^*=
\begin{cases}
\theta\cdot(\chi^{(1^3)})^m  &\text{ if } x=0\,,\\

\theta\cdot(\chi^{(2,1)}) &\text{ if } x=1\,,\\

\theta\cdot(\chi^{(1^3)})^{m+1} &\text{ if } x=2\,.\\
\end{cases}
\end{eqnarray*}
\end{corollary}

\begin{proof}
If $x=1$ the statement follows immediately from Theorem \ref{t:main correspondence p^k}.
Suppose that $x=0$. By Theorem \ref{t:main correspondence p^k} we know that $\chi^*=\theta\cdot\chi^{\nu}$, 
for some $\nu\in\{(3), (1^3)\}$. 
Adopting the notation of Lemma \ref{l:231}, let $\alpha=(2\cdot 3^{k-1}-2m,1^{2m})\in\mathcal{H}(2\cdot 3^{k-1})$ and let $\psi=\chi^\alpha\otimes\color{black}\chi^\mu \in\mathrm{Irr}_{3'}(A)$, where $A=\fS_{2\cdot 3^{k-1}}\times\fS_{3^{k-1}}$. 
From Proposition \ref{p:hooks23^k} we deduce that $\Delta:=(\sigma\cdot(\chi^{(1^2)})^{m})\otimes \chi^\mu$ is a constituent of $\psi_K$, where $K:=\big(\fS_{3^{k-1}}\wr \fS_2\big)\times \fS_{3^{k-1}}$  and $\sigma$ is the canonical extension of
$\chi^\mu\otimes\chi^\mu$ to its stabilizer given by Lemma \ref{wr-ext}\color{black}.  
Moreover, $\Delta$ is the unique constituent of $\chi^\lambda\downarrow_{K}$ lying over $\theta$, by Lemma \ref{lem:1}. Therefore $\chi^*\downarrow_K=\Delta$ and hence 
$\chi^\nu\downarrow_{\fS_2}=(\chi^{(1^2)})^m$ (as characters of $(\fS_{3^{k-1}}\wr \fS_3)/(\fS_{3^{k-1}})^3\cong \fS_3$). 
This implies that $\chi^\nu=(\chi^{(1^3)})^m$, as desired. 
By Theorem \ref{t:main correspondence p^k}, this also settles the case $x=2$. 
\end{proof}

\subsection{Arbitrary $n\in\mathbb{N}$}\label{sec:77}
In this section we let $n=\sum_{k=1}^ta_k\cdot 3^{n_k}$ be the $3$-adic expansion of $n$, for some $n_1>\cdots >n_t\in\mathbb{N}_0$ and some $a_k\in\{1,2\}$ for all $k\in\{1,2,\ldots, t\}$.
We relax a bit the notation by saying that a partition $\lambda\in\mathrm{Irr}_{3'}(\fS_n)$ or equivalently that $\lambda\vdash_{3'} n$ if the corresponding irreducible character $\chi^\lambda\in\mathrm{Irr}_{3'}(\fS_n)$.

\begin{proposition} \label{prop:Olsson}
 Let $\lambda$ be a partition of $n$. Then $\lambda\in\mathrm{Irr}_{3'}(\fS_n)$ if and only if one of the following conditions holds. 
\begin{itemize}
\item[(i)] $\lambda$ has a unique removable $(a_1\cdot 3^{n_1})$-hook $h_1$, and $\lambda_{(a_1\cdot 3^{n_1})}\in\mathrm{Irr}_{3'}(\fS_{n-a_1\cdot 3^{n_1}})$.
\item[(ii)] $a_1=2$, $\lambda$ has two removable $3^{n_1}$-hooks $h_1, h_2$ and $\lambda_{(3^{n_1})}\in\mathrm{Irr}_{3'}(\fS_{n-2\cdot 3^{n_1}})$.
\end{itemize}
\end{proposition}
\begin{proof}
The statement follows directly from Theorem \ref{Mac}, applied to the case where $p=3$. 
\end{proof}

\begin{remark}\label{remark: 3tower}
Let $m\in\mathbb{N}$  be such that $m=\sum_{k=1}^ta_k\cdot 3^{m_k}$ is the $3$-adic expansion of $m$, for some $m_1>m_2\cdots >m_t\geq 0$. 
Let $N\in\mathbb{N}$ be such that $N>m_1$ and let $n=3^N+m$.

Let $\gamma\in\mathcal{P}(m)$ and $\ell\in\{0,1,\ldots, 3^{N}-1\}$. From \cite[Theorem 1.1]{Bess} we know that $\gamma$ has exactly one addable $3^N$-rim hook of leg length $\ell$. Hence $\gamma$ has exactly $3^N$ addable $3^N$-hooks (one for every possible leg-length). 
Let $\lambda_1,\lambda_2,\ldots, \lambda_{3^N}$ be the distinct partitions of $n$ obtained by adding a $3^N$-hook to $\gamma$.

For each $i\in\{1,\ldots, 3^N\}$ the $3$-core tower $T(\lambda_i)$ has the following form. $T_j(\lambda_i)=T_j(\gamma)$ for all $j<N$ and there exists a permutation $\sigma\in\fS_{3^N}$ such that $T_N(\lambda_i)=((\lambda_i)_1^N,\ldots, (\lambda_i)_{3^N}^N)$, where 
$(\lambda_i)^N_{\sigma(i)}=(1)$ and $(\lambda_i)_s^N=\emptyset$ for all $s\in\{1,\ldots, 3^{N}\}\smallsetminus\{\sigma(i)\}$.


In particular, the position of the non-empty $3$-core in the $N$-th layer $T_N(\lambda)$ of the $3$-core tower $T(\lambda)$ uniquely determines the leg-length of the unique removable $3^N$-hook of $\lambda_i$.
\end{remark}

Proposition \ref{prop:Olsson} allows us to associate to each partition $\lambda\in\mathrm{Irr}_{3'}(\fS_n)$ a sequence of partitions $\lambda^*=(\mu_1,\mu_2,\ldots, \mu_{t})\in \mathrm{Irr}_{3'}(\fS_{a_1\cdot 3^{n_1}}\times \cdots \times \fS_{a_t\cdot 3^{n_t}})$, in the following way.

\smallskip

\textbf{(i)} If $\lambda$ has a unique removable $(a_1\cdot 3^{n_1})$-hook $h_1$, then let $\mu_1=h_1$ and set $\lambda_2=\lambda_{(a_1\cdot 3^{n_1})}=\lambda- h_1$.

\smallskip

\textbf{(ii)} If $a_1=2$ and $\lambda$ does not have a removable $(a_1\cdot 3^{n_1})$-hook, then by Proposition \ref{prop:Olsson} it has two removable $3^{n_1}$-hooks $h_1, h_2$. For each $j\in\{1,2\}$ let $\gamma_j:=\lambda- h_j$. Clearly $\gamma_j$ has a unique removable $3^{n_1}$-hook $k_j$.
In this case let $\mu_1\in\mathrm{Irr}_{3'}(\fS_{2\cdot 3^{n_1}})$ be the unique partition hook-generated by $\{k_1,k_2\}$ and let 
$\lambda_2=\lambda_{(3^{n_1})}$.
Notice that this is well defined since $k_1\neq k_2$. This can be observed by applying Remark \ref{remark: 3tower} to the partition $\gamma=\lambda_{(3^{n_1})}$.

\medskip

Either ways we have $\mu_1\in\mathrm{Irr}_{3'}(\fS_{a_1\cdot 3^{n_1}})$ and $\lambda_2\in\mathrm{Irr}_{3'}(\fS_{n-a_1\cdot 3^{n_1}})$, so we can reapply the process to $\lambda_2$ and obtain $\mu_2\in\mathrm{Irr}_{3'}(\fS_{a_2\cdot 3^{n_2}})$ and $\lambda_3\in\mathrm{Irr}_{3'}(\fS_{n-a_1\cdot 3^{n_1}-a_2\cdot 3^{n_2}})$. 
After $t$ iterations of this algorithm we obtain the desired sequence $\lambda^*=(\mu_1,\mu_2,\ldots, \mu_{t})$.

\begin{proposition}\label{prop:arb}
The map $\Gamma:\lambda\mapsto \lambda^*=(\mu_1,\mu_2,\ldots, \mu_{t})$ is a bijection between $\mathrm{Irr}_{3'}(\fS_n)$ and $\mathrm{Irr}_{3'}(\fS_{a_1\cdot 3^{n_1}}\times \cdots \times \fS_{a_t\cdot 3^{n_t}})$.
Moreover, if $\mu_j\in\mathcal{H}(a_j\cdot  3^{n_j})$ for all $j\in\{1,\ldots t\}$ then $\chi^{\lambda^*}$ is an irreducible constituent of $\chi^\lambda\downarrow_{\fS_{a_1\cdot 3^{n_1}}\times \cdots \times \fS_{a_t\cdot 3^{n_t}}}$.
\end{proposition}
\begin{proof}
The discussion before the statement of Proposition \ref{prop:arb} shows that each $\lambda\in\mathrm{Irr}_{3'}(\fS_n)$ uniquely determines a sequence 
$(\mu_1,\mu_2,\ldots, \mu_{t})\in \mathrm{Irr}_{3'}(\fS_{a_1\cdot 3^{n_1}}\times \cdots \times \fS_{a_t\cdot 3^{n_t}})$.
We show now that this map has an inverse. This is clearly enough to prove the proposition. 

We proceed by induction on the length $t$ of the $3$-adic expansion of $n$. If $t=1$ then the statement holds trivially.
Suppose that $t\geq 2$. 
Let $(\mu_1,\mu_2,\ldots, \mu_{t})\in \mathrm{Irr}_{3'}(\fS_{a_1\cdot 3^{n_1}}\times \cdots \times \fS_{a_t\cdot 3^{n_t}})$ and let $\gamma\in\mathrm{Irr}_{3'}(\fS_{n-a_1\cdot 3^{n_1}})$ be the partition corresponding to the sequence $(\mu_2,\ldots, \mu_{t})$. 
The partition $\gamma$ exists and it is well defined by inductive hypothesis. 
Since $\mu_1\in\mathrm{Irr}_{3'}(\fS_{a_1\cdot 3^{n_1}})$ we have that $(\mu_1,\gamma)$ identifies a unique partition $\lambda\in\mathrm{Irr}_{3'}(\fS_n)$ as follows (we need to distinguish two cases, depending on the shape of $\mu_1$). 

\smallskip

\textbf{(i)} Suppose that $\mu_1=(a_1\cdot 3^{n_1}-\ell,1^{\ell})\in\mathcal{H}(a_1\cdot 3^{n_1})$. By \cite[Theorem 1.1]{Bess}, we know that $\gamma$ has a unique addable $(a_1\cdot 3^{n_1})$-rim hook $h_1$ of leg length $\ell$. In this case $\lambda$ is defined as the partition obtained by adding $h_1$ to $\gamma$. 
(More formally $\lambda$ is the unique partition of $n$ that has a removable $(a_1\cdot 3^{n_1})$-hook $h_1$ of leg length $\ell$, and such that $\lambda- h_1=\gamma$).

\smallskip

\textbf{(ii)} Suppose that $a_1=2$ and $\mu_1\in\mathrm{Irr}_{3'}(\fS_{a_1\cdot 3^{n_1}})$ is not a hook partition. Then there exists $\{k_1,k_2\}\subseteq\mathcal{H}(3^{n_1})$ , $k_1\neq k_2$ such that $\{k_1,k_2\}$ hook-generates $\mu_1$. For $j\in\{1,2\}$ let $\ell_j$ be the leg-length of $k_j$. Moreover, denote by $\nu_j$ the (unique) partition of $n-3^{n_1}$ obtained by adding (in the unique possible way) a $3^{n_1}$-hook of leg length $\ell_j$ to $\gamma$.
In particular, looking at the $3$-core towers of $\nu_1$ and $\nu_2$ we have that there exist $x_1, x_2\in\{1,2,\ldots, 3^{n_1}\}$ such that:  
$T_{n_1}(\nu_j)=((\nu_j)_1^{n_1},\ldots, (\nu_j)_{3^{n_1}}^{n_1})$, where $(\nu_j)_{x_j}^{n_1}=(1)$ and $(\nu_j)_{i}^{n_1}=\emptyset$ for all $i\in\{1,2,\ldots, 3^{n_1}\}\smallsetminus\{x_j\}$.
From Remark \ref{remark: 3tower}, we have that $x_1\neq x_2$. 

We now define $\lambda$ as the partition of $n$ such that $T_j(\lambda)=T_j(\gamma)$ for all $j<n_1$ and $T_{n_1}(\lambda)=(\lambda_1^{n_1},\ldots, \lambda_{3^{n_1}}^{n_1})$, where $\lambda_i^{n_1}=(1)$ for $i\in\{x_1,x_2\}$ and $\lambda_i^{n_1}=\emptyset$ for $i\in\{1,2,\ldots, 3^{n_1}\}\smallsetminus\{x_1,x_2\}$.

\medskip

In both cases we obtain a canonically defined partition $\lambda\in\mathcal{P}(n)$ such that $\chi^\lambda\in\mathrm{Irr}_{3'}(\fS_n)$ and such that 
$\lambda^*=(\mu_1,\mu_2,\ldots, \mu_{t})$. 

The second statement is proved by induction on $t$. If $t=1$ the statement holds trivially. Let $t\geq 2$ and let $h$ be the unique removable $a_13^{n_1}$-hook of $\lambda$. Then $\lambda^*=(h, \mu_2,\ldots,\mu_t)$, $(\lambda- h)^*=(\mu_2,\ldots,\mu_t)$ and 
$$\left\langle\chi^{\mu_2}\otimes\cdots\otimes\chi^{\mu_t}\uparrow^{\fS_{n-|h|}}, \chi^{\lambda- h}\right\rangle\neq 0,$$
by inductive hypothesis. Moreover, an easy application of the Littlewood-Richardson rule shows that $\left\langle(\chi^{\lambda- h}\otimes \chi^{h})\uparrow^{\fS_{n}}, \chi^\lambda\right\rangle\neq 0$.
This concludes the proof.
\end{proof}

Proposition \ref{prop:arb} together with Theorems \ref{t:2main correspondence p^k} and \ref{t:main23^k} implies Theorem A, which 
we reformulate below: 

\begin{thm}\label{thm:maintotal}
Let $n=\sum_{k=1}^ta_k\cdot 3^{n_k}$ be the $3$-adic expansion of $n$. The map $\Phi$ defined as the composition of the maps $\Gamma$ and $(\Phi_{a_1\cdot 3^{n_1}}\times\cdots\times\Phi_{a_t\cdot 3^{n_t}})$
is a canonical bijection between $\mathrm{Irr}_{3'}(\fS_n)$ and $\mathrm{Irr}_{3'}(\NB_{\fS_n}(P_n))$.
Moreover, if $a_k\neq 2$ for all $k\in\{1,\ldots, t\}$ and $\chi\in\mathrm{Irr}_{3'}(\fS_n)$, then $\Phi(\chi)$ is an irreducible constituent of $\chi\downarrow_{\NB_{\fS_n}(P_n)}$.
\end{thm}

Of course, in the previous theorem we embed 
$\fS_{a_1\cdot 3^{n_1}}\times\cdots\times\fS_{a_t\cdot 3^{n_t}}$ naturally as a Young subgroup of $\fS_n$. 
We also note that the characters in $\irrq{\norm{\fS_n}{P_n}}{3}$ are rational-valued, as it can be checked by using Lemma \ref{norm} below, for instance. 

\section{Irreducible characters of $3'$-degree of general linear and unitary groups}\label{sec:8}
Using the McKay bijection previously described for symmetric groups, it is possible to construct a similar canonical bijection for general linear and unitary groups in non-defining characteristic. 
From now on let $q$ be a power of a prime $p\neq 3$. 

\subsection{General linear groups}
For any $n \geq 1$, let $G = \GL_n(q)$ be the finite general linear group, with a natural 
module $V = \mathbb{F}_q^n = \left\langle e_1, \ldots ,e_n \right\rangle_{\mathbb{F}_q}$. As in \cite{KT2}, it is convenient for us to use 
the Dipper-James classification of complex irreducible characters of $G$, as described in \cite{JamesGL}. Namely, every $\chi\in \mathrm{Irr}(G)$ can be written uniquely, up to a permutation of the pairs 
$\{(s_1,\lambda_1),\dots,(s_m,\lambda_m)\}$, in the form 
\begin{equation}\label{gl1}
  \chi = S(s_1,\lambda_1) \circ S(s_2,\lambda_2) \circ \ldots \circ S(s_m,\lambda_m).
\end{equation}
Here, $s_i \in \overline{\mathbb{F}}_q^\times$ has degree $d_i$ over $\mathbb{F}_q$, $\lambda_i \vdash k_i$, 
$\sum^m_{i=1}k_id_i = n$, and the $m$ elements $s_i$ have pairwise distinct minimal polynomials over $\mathbb{F}_q$. 
In particular, $S(s_i,\lambda_i)$ is a primary irreducible character of $\GL_{k_id_i}(q)$. 
The subset $\mathrm{Irr}_{3'}(G)$ is conveniently described by Lemmas (5.3) and (5.4) of \cite{Olsson}. In what follows, we will
fix a basis of $V$ and use this basis to define the field automorphism $F_p : (x_{ij}) \mapsto (x_{ij}^p)$ and 
the transpose-inverse automorphism $\tau : X \mapsto \tw tX^{-1}$; also set $D^+ := \langle F_p,\tau \rangle$. Note that
$\Aut(G) = \Inn(G) \rtimes D^+$. Then, in this and the subsequent sections, our canonical bijections will be 
shown to be equivariant under $D^+$ and also $\Gamma := \Gal(\bar\QQ/\QQ)$.

\subsection{General unitary groups}

Let $G^- := \GU_n(q) = \GU(V)$, the group of unitary transformations of
the Hermitian space $V = \FF^n_{q^2}$. 
The {\it Ennola duality} establishes a procedure to obtain $\Irr(G^-)$ from $\Irr(G^+)$ by a formal change 
$q$ to $-q$, where $G^+ := \GL_n(q)$. We will exhibit it more explicitly for our purposes of studying $\Irr_{3'}(G^-)$.

As in the case of $G^+$, we can identify $G^-$ with its dual group (in the sense of the Deligne-Lusztig theory \cite{C}, \cite{DM}). 
It is easy to see that any semisimple element $s_- \in G^-$ has centralizer of the form
$$\CB_{G^-}(s_-) \cong \GU_{k_1}(q^{d_1}) \times \ldots \times \GU_{k_a}(q^{d_a}) \times \GL_{k_{a+1}}(q^{d_{a+1}}) \times \ldots \times
    \GL_{k_r}(q^{d_r}),$$
where $\sum^r_{i=1}k_id_i  = n$, all $d_1, \ldots,d_a$ are {\it odd}, and all $d_{a+1}, \ldots ,d_r$ are {\it even}. 
In what follows, we use the convention 
$$\GL_k(-q^d) := \GU_k(q^d)$$
for any positive integers $k,d$. Then we can write 
$$\CB_{G^-}(s_-) \cong \GL_{k_1}((-q)^{d_1}) \times \ldots \times \GL_{k_a}((-q)^{d_a}) \times \GL_{k_{a+1}}((-q)^{d_{a+1}}) \times \ldots \times
    \GL_{k_r}((-q)^{d_r});$$
which can then be obtained, replacing $q$ to $-q$, from
$$\CB_{G^+}(s_+) \cong \GL_{k_1}(q^{d_1}) \times \ldots \times \GL_{k_r}(q^{d_r})$$
for some semisimple element $s_+ \in G^+$. Each unipotent character of $\GL_{k_i}((-q)^{d_i})$ is labeled 
by a partition $\lam_i \vdash k_i$; let $\psi^-(\lam_i)$ denote such character and set 
$$\psi^- = \psi^-(\lam_1) \otimes \psi^-(\lam_2) \otimes \ldots \otimes \psi^-(\lam_r).$$
Note, see \cite[Chapter 13]{C}, that if $\psi^+(\lam_i)$ is the unipotent character of $\GL_{k_i}(q^{d_i})$
labeled by the same $\lam_i$, then $\deg \psi^+(\lam_i)$ is a product of some powers of $q$ and some cyclotomic
polynomials in $q$; furthermore, if we make the formal change $q$ to $-q$, then we obtain $\deg \psi^-(\lam_i)$
(up to sign):
$$\deg \psi^-(\lam_i) = \pm \deg \psi^+(\lam_i)|_{q \to -q}.$$
As explained on \cite[p. 116]{FS}, for each $\chi^- \in \Irr(G^-)$, we can find 
a unique $G^-$-conjugacy class of semisimple elements $s_- \in \Irr(G^-)$, a linear character $\hat{s}_-$ of $L_- := \CB_{G^-}(s_-)$ and 
$\psi^-$ as above so that $\chi^- = \pm R^{G^-}_{L_-}(\hat{s}_-\psi^-)$, where $R^{G^-}_{L-}$ is the {\it Lusztig induction}. Similarly, 
we can find a linear character $\hat{s}_+$ of $L_+: =\CB_{G^+}(s_+)$ and form $\chi^+ = R^{G^+}_{L_+}(\hat{s}_+\psi^+) \in \Irr(G^+)$,
with 
$$\psi^+ = \psi^+(\lam_1) \otimes \psi^+(\lam_2) \otimes \ldots \otimes \psi^+(\lam_r).$$
Note that 
$$\chi^+ = S(s'_1,\lam_1) \circ S(s'_2,\lam_2) \circ \ldots \circ S(s'_r,\lam_r)$$
for suitable $s'_i \in \bar\FF_q^\times$ of degree $d_i$ in \eqref{gl1}. Adopting this notation, we will also write $\chi^-$ in the form
\begin{equation}\label{gu1}
  \chi^- = S(s_1,\lam_1) \circ S(s_2,\lam_2) \circ \ldots \circ S(s_r,\lam_r),
\end{equation} 
with the following modification: either $d_j \geq 3$ and $s_j \in \bar\FF_q^\times \smallsetminus \FF_{q^2}$,
or $d_j=1$ and $s_j^{q+1}=1$, or $d_j = 2$, and $s_j \in \FF_{q^2}^\times$ but $s_j^{q+1} \neq 1$ (note that each $s_j$ is 
an eigenvalue of $s_-$ on $V$ corresponding to the factor $\GL_{k_j}((-q)^{d_j})$ of $\CB_{G^-}(s_-)$). In the two latter cases,
we define 
$$\deg^-(s_j) := d_j.$$

According to \cite[(1.3)]{FS},
$$\chi^+(1) = [G^+:L_+]_{p'} \cdot \psi^+(1),~~\chi^-(1) = [G^-:L_-]_{p'} \cdot \psi^-(1).$$
Furthermore, $|G^-|_{p'}$, respectively $|L_-|_{p'}$, can be obtained from $|G^+|_{p'}$, respectively from $|L_+|_{p'}$, 
by replacing $q$ with $-q$ (and possibly multiplying by $-1$ to get a positive number!) It follows that $\chi^-(1)$ is obtained from 
$\chi^+(1)$ by the formal change $q \to -q$. More precisely, there is a monic polynomial $f = \sum^D_{i=0}c_it^i\in \ZZ[t]$ in the variable $t$ (which 
is a product of a power of $t$ and some cyclotomic polynomials in $t$, independent from $q$), such that 
$$\chi^+(1) = f(q),~~\chi^-(1) = \pm f(-q).$$

Now we choose $e^- \in \{1,2\}$ such that the order of $(-q)$ modulo $3$ is $e^-$, and choose
$e^+ \in \{1,2\}$ to be the order of $q$ modulo $3$. First consider the case $e^-=1$, i.e.
$3|(q+1)$. Then $3 \nmid \chi^-(1)$ {\it with} $e^- = 1$ if and only if $3 \nmid f(-q) = \sum_ic_i(-q)^i \equiv \sum_i c_i (\mod 3)$,  
i.e. precisely when $3 \nmid \chi^+(1) = f(q)$ {\it with} $e^+=1$. Next, we consider the case $e^-=2$, i.e.
$3|(q-1)$. Then $3 \nmid \chi^-(1)$ {\it with} $e^- = 2$ if and only if $3 \nmid f(-q) = \sum_ic_i(-q)^i \equiv \sum_i (-1)^ic_i (\mod 3)$,  
i.e. precisely when $3 \nmid \chi^+(1) = f(q)$ {\it with} $e^+=2$. In each of these cases, we can then apply Lemmas (5.3) and (5.4) of
\cite{Olsson} to $\chi^+$ to get the conditions on $d_j = \deg(s'_j) = \deg^-(s_j)$ and $\lam_j$. 
Hence we have obtained (where again $\mu_{(e)}$ and $\mu^{(e)}$ denote
the $e$-core and the $e$-power of a partition $\mu$):

\begin{thm}\label{char-gu}
Write $n=c+me^-$ with $0 \leq c < e^-$, where $e^- \in \{1,2\}$ is the order of $(-q)$ modulo $3$. Then the character $\chi^-$ in
\eqref{gu1} has degree coprime to $3$ if and only if all the following conditions hold:

\begin{enumerate}[\rm(i)]
\item If $n_j := k_jd_j = c_j + m_je^-$ with $0 \leq c_j < e^-$ and $1 \leq j \leq r$, then $\sum^r_{j=1}c_j = c$,
$\sum^r_{j=1}m_j = m$, and 
$3 \nmid (m!/\prod^r_{j=1}m_j!)$.
\item $s_j \in \FF_{q^2}^\times$, and $d_j = \deg^-(s_j)$ divides $e^-$, for $1 \leq j \leq r$. 
\item Set $e'_j := e^-/d_j$ for $1 \leq j \leq r$. Then $d_j|(\lam_j)_{(e'_j)}| = c_j$, $(\lam_j)^{(e'_j)} = (\lam_{j,0}, \ldots ,\lam_{j,e'_j-1})$, with 
$\lam_{j,i} \vt m_{j,i}$ for $0 \leq i \leq e'_j-1$. Next, $\sum^{e'_j-1}_{i=0}m_{j,i} = m_j$ and 
$3 \nmid (m_j!/\prod^{e'_j-1}_{i=0}m_{j,i}!)$.
\end{enumerate} 

\end{thm}

Again, in what follows we will fix a basis of $V$ and use this basis to define the field automorphism $F_p : (x_{ij}) \mapsto (x_{ij}^p)$; 
also set $D^-:= \langle F_p \rangle$. Note that
$\Aut(G) = \Inn(G) \rtimes D^-$. Then, in this and the subsequent sections, our canonical bijections will be 
shown to be equivariant under $D^-$ and also $\Gamma = \Gal(\bar\QQ/\QQ)$.

\begin{remark}\label{rem:actionaut}
In both of the cases $\eps = \pm$, the action of $D^\eps$ and $\Gamma$ on the labels of $\Irr(G^\eps_n)$ is explained in the proof
of \cite[Theorem 5.3]{GKNT}. In particular, $S(s,\lam)^\sigma = S(s^\sigma,\lam)$ and 
$S(s,\lam)^{F_p} = S(s^p,\lam)$,  for $\sigma \in \Gamma$ and $F_p \in D^\eps$. Moreover $S(s,\lam)^\tau=S(s^{-1},\lam)$ for $\tau\in D^+$.
(The action of $\Gamma$ on $\overline{\mathbb F}_q^\times$ is defined at the beginning of Section 5 in \cite{GKNT}.)
\end{remark}

\section{Linear and Unitary Groups: A Global Bijection}

Let $q$ be a power of a prime $p\neq 3$ and 
let $G_n^\eps$ be $\GL_n(q)$ for $\eps = +$
and $\GU_n(q)$ for $\eps=-$. When necessary, we will write $G_n^\eps(q)$ instead of the shorter notation $G_n^\eps$. We 
will also interpret the expressions like $q-\eps$ as $q-\eps 1$. 

Let $R^\eps$ be a Sylow $3$-subgroup of $G_n^\eps$. In order to explicitly describe a canonical bijection between $\mathrm{Irr}_{3'}(G_n^\eps)$ and $\mathrm{Irr}_{3'}(\NB_{G_n^\eps}(R^\eps))$, we will adopt a strategy that is similar to the one used in the previous sections, in the case of symmetric groups. Namely, we will identify a convenient subgroup $H^\eps$ such that $\NB_{G_n^{\eps}}(R^\eps)\leq H^\eps\leq G_n^\eps$ and we will first construct canonical bijection between $\mathrm{Irr}_{3'}(G_n^\eps)$ and $\mathrm{Irr}_{3'}(H^\eps)$, and a second one between $\mathrm{Irr}_{3'}(H^\eps)$ and $\mathrm{Irr}_{3'}(\NB_{G_n^\eps}(R^\eps))$. Let 
$V^+ = \FF_q^n = \langle e_1, \ldots ,e_n\rangle$ denote the natural module for $G^+$, and $V ^- = \FF_{q^2}^n$, endowed with
an orthonormal basis $e_1, \ldots ,e_n$, be the natural module for $G^-$. 

\subsection{The case $3\mid (q-\eps)$}\label{q-1}
Suppose $3$ divides $q-\eps$. 
Let $n\in\mathbb{N}$ and let $\Lambda_{3'}(n)$ be the set consisting of all $(n_1,n_2,\ldots, n_r)\vdash n$ such that $$3\nmid \frac{n!}{\prod_{j=1}^r(n_j!)}.$$
From \cite[Section 5]{Olsson} and the discussions in Section \ref{sec:8} we deduce that $\chi\in\mathrm{Irr}(G_n^\eps)$ has degree coprime to $3$ if and only if there exists a partition $(n_1,n_2,\ldots, n_r)\in\Lambda_{3'}(n)$ and pairwise distinct elements $s_1,\ldots, s_r\in\overline{\mathbb{F}}_q^\times$ 
with $s_j^{q-\eps}=1$ for all $j$, 
such that 
$$\chi=S(s_1,\lambda_1) \circ S(s_2,\lambda_2) \circ \ldots \circ S(s_m,\lambda_m),$$
where $\chi^{\lambda_j}\in\mathrm{Irr}_{3'}(\fS_{n_j})$, for all $j\in\{1,\ldots, r\}$. 

Then we define the subgroup
$H^\eps$  to be $\GL_1(q) \wr \fS_n$ 
when $\eps = +$, and 
$\GU_1(q) \wr \fS_n$
when $\eps = -$, 
where in both cases $\fS_n$ denotes the subgroup of permutation matrices in 
$G^\varepsilon_n$ with respect to our fixed basis $e_1,\ldots, e_n$ of $V^\varepsilon$. \color{black} Note that 
$H^\eps$ is invariant under the field automorphism $F_p: (x_{ij}) \mapsto (x_{ij}^p)$ (defined in the given basis of $V^\eps$), 
and also under the transpose-inverse automorphism $\tau:X \mapsto \tw t X^{-1}$ when $\eps = +$. 
Furthermore, both $\Gamma$ and $D^\varepsilon$ act trivially on $\fS_n$. 
\color{black}
The indicated automorphisms $\tau$ (for $\eps=+$) and $F_p$, together
with the inner automorphisms, generate the full group $\Aut(G_n^\eps)$. 
Note that $H^\eps=G^\eps_{1}\wr \fS_n$. An easy application of the theory recalled in Section \ref{pre_ch} shows that $\psi\in\mathrm{Irr}(H^\eps)$ has degree coprime to $3$ if and only if there exists a partition $\rho:=(n_1,n_2,\ldots, n_r)\in\Lambda_{3'}(n)$ and pairwise distinct elements $s_1,\ldots, s_r\in\overline{\mathbb{F}}_q^\times$ 
such that $s_j^{q-\eps}=1$ for all $j$, 
such that 
$$\psi=\big(\widetilde{(S(s_1,(1))}^{\otimes n_1}\cdot\chi^{\lambda_1})
\otimes \widetilde{(S(s_2,(1))}^{\otimes n_2}\cdot\chi^{\lambda_2})
\otimes\cdots\otimes\widetilde{(S(s_r,(1))}^{\otimes n_r}\cdot\chi^{\lambda_r})\big)\big\uparrow_{I_{\rho}}^{H^\eps},$$
where $\lambda_j\vdash_{3'} n_j$, for all $j\in\{1,\ldots, r\}$. 
Here we denoted by $I_\rho$ the subgroup of $H^\eps$ defined by
$I_\rho=G^\eps_{1}\wr (\fS_{n_1}\times\cdots\times\fS_{n_r})=(G^\eps_{1}\wr \fS_{n_1})\times\cdots\times (G^\eps_{1}\wr \fS_{n_r})\leq H^\eps$.
Moreover, for all $1\leq i\leq r$ the character $\widetilde{S(s_i,(1))}^{\otimes n_i}\in\mathrm{Irr}(G^\eps_{1}\wr \fS_{n_i})$ denotes the extension of $S(s_i,(1))^{\otimes n_i}\in\mathrm{Irr}((G^\eps_{1})^{n_i})$, prescribed by Lemma \ref{wr-ext}.

The above discussion shows that the following statement holds. 

\begin{proposition}\label{prop q-1}
Let $\chi=S(s_1,\lambda_1) \circ S(s_2,\lambda_2) \circ \ldots \circ S(s_r,\lambda_r)\in\mathrm{Irr}_{3'}(G_n^\eps).$
Denote by $\chi^\star$ the element of $\mathrm{Irr}_{3'}(H^\eps)$ defined by 
$$\chi^\star=\big(\widetilde{(S(s_1,(1))}^{\otimes n_1}\cdot\chi^{\lambda_1})
\otimes \widetilde{(S(s_2,(1))}^{\otimes n_2}\cdot\chi^{\lambda_2})
\otimes\cdots\otimes\widetilde{(S(s_r,(1))}^{\otimes n_r}\cdot\chi^{\lambda_r})\big)\big\uparrow_{I_{\rho}}^{H^\eps},$$
where $\rho=(|\lambda_1|,\ldots,|\lambda_r|)$. 
The map $\chi\mapsto\chi^\star$ is a canonical bijection between $\mathrm{Irr}_{3'}(G_n^\eps)$ and $\mathrm{Irr}_{3'}(H^\eps)$. 
Moreover, the bijection is both $\Gamma$-equivariant and $D^\eps$-equivariant. 
\end{proposition}
\begin{proof}
The discussion in Section \ref{q-1} easily implies that the map $\chi\mapsto\chi^\star$ is a bijection. 
Moreover, using Lemma \ref{wr-ext} together with Remark \ref{rem:actionaut} it is routine to check that the bijection is both $\Gamma$-equivariant and $D^\eps$-equivariant. 
\end{proof}

\subsection{The case $3|(q+\eps)$}\label{q+1}
Let $n=2m$ (we will treat the case $n=2m+1$ separately, see Theorem \ref{odd-even}). 
We will explicitly construct a subgroup $H^\eps = H^\eps_m \cong (\GL_1(q^2)\rtimes_{\eps} C_2)\wr \fS_m$ that is 
invariant under the field automorphism $F_p$, and also under the transpose-inverse automorphism $\tau$ if $\eps = +$.
We start by giving the construction in the case $m = 1$. 
We let $V_1$ be the natural module for $G^\eps_1$.

First let $\eps = -$. Then we fix a basis $(e,f)$ of $V_1$ such that the Hermitian form takes values
$e \circ e = f \circ f = 0$ and $e\circ f = 1$. Then we take $\GL_1(q^2)$ to be the subgroup
$\{\diag(a,a^{-q}) \mid a \in \FF_{q^2}^\times\}$, and the involution $z$ to be swapping $e$ and $f$. Now if $F_p$ is 
defined in this basis $(e,f)$ then $H^\eps_1 = \langle \GL_1(q^2),z \rangle$ is $F_p$-invariant.

Next suppose $\eps = +$. Note that $3|(q+1)$, and so $q=p^f$ for some odd $f$ and $3|(p+1)$. Furthermore, we can
find $\beta \in \FF_{p^2} \smallsetminus \FF_q$ of order $(p+1)\cdot\gcd(2,p-1)$, so that 
$\beta^{p+1} = -1 = \beta^{q+1}$. Now we view $\FF_{q^2}$ as a vector space over $\FF_q$ with basis $(1,\beta)$ and let
$M_x$ denote the matrix of the multiplication $M_x(v) = xv$ by $x \in \FF_{q^2}^\times$, relative to this basis. Note that 
$\gamma:=\beta+\beta^p \in \FF_p$, and $M_{a+b\beta} = \begin{pmatrix}a & b\\b & a+\gam b \end{pmatrix}$ for 
$a,b \in \FF_q$. It follows that $F_p(M_{a+b\beta}) = M_{a^p+b^p\beta}$ and 
$\tau(M_x) = M_{x^{-1}}$. Then we take $z$ to be (the action of) $F_p$, which, by the above discussion, is represented by 
$M_z := \begin{pmatrix}1 & \gam \\0 & -1\end{pmatrix}$. In particular, $M_z$ is $F_p$-fixed and 
$\tau(M_z) = M_{1+\gam\beta}M_z$. Hence, $H^\eps_1 = \langle M_x,M_z \mid x \in \FF_{q^2}^\times \rangle$ is 
$\langle F_p,\tau \rangle$-invariant.

In the general case for any $m$, it suffices to take $V$ to be the sum of $m$ isomorphic copies of $V_1$ (with our fixed basis), and 
set $H^\eps_m = H^\eps_1 \wr \fS_m\cong (\GL_1(q^2)\rtimes_\eps C_2)\wr \fS_m$, 
where $\fS_m$ denotes the subgroup of permutation matrices in $G^\varepsilon_{2m}$ acting on the diagonal $(2\times 2)$-blocks
of the matrices in $(H^\eps_1)^{2m}$.
\color{black}
Note that the action of $ C_2=\langle z\rangle$ on $\GL_1(q^2)$ is given by exponentiation to the $(\eps q)^{th}$ power. 
In particular, the subgroup of the fixed points under this action coincides with $G^\eps_1\leq \GL_1(q^2)$.

An irreducible character $\chi$ of $H^\eps$ has degree coprime to $3$ if and only if there exists a partition $\rho=(m_1,\ldots, m_r)\in\Lambda_{3'}(m)$ such that
$$\chi=\theta_1\otimes\theta_2\otimes\cdots\otimes\theta_r\big\uparrow_{I_\rho}^{H^\eps},$$
where $I_\rho=(\GL_1(q^2)\rtimes_\eps C_2)\wr \fS_\rho=[(\GL_1(q^2)\rtimes_\eps C_2)\wr \fS_{m_1}]\times \cdots\times [(\GL_1(q^2)\rtimes_{\eps} C_2)\wr \fS_{m_r}],$
and $\theta_j:=(\widetilde{\psi_j^{\otimes m_j}})\cdot\chi^{\lambda_j}$. Here $\psi_1,\ldots, \psi_r$ are pairwise distinct irreducible characters of $\GL_1(q^2)\rtimes_\eps C_2$, $\widetilde{\psi_j^{\otimes m_j}}$ denotes the extension of
$\psi_j^{\otimes m_j}$ to $(\GL_1(q^2)\rtimes_\eps C_2)\wr \fS_{m_j}$ given in Lemma \ref{wr-ext}, 
and $\lambda_j$ is a partition of $m_j$ such that $\chi^{\lambda_j}\in\mathrm{Irr}_{3'}(\fS_{m_j})$, for all $j\in\{1,\ldots, r\}$. 
Moreover, we denoted by $\fS_\rho$ the Young subgroup of $\fS_m$ corresponding to $\rho\in\mathcal{P}(m)$.

From now on we denote the group $\GL_1(q^2)\rtimes_\eps C_2$ by $K^\eps$. 
For each $g\in G_1^\varepsilon$, note that $S(g, (1))\in\irr{\GL_1(q^2)}$ is invariant in $K^\varepsilon$. 
Since $\GL_1(q^2)$ is cyclic, it is not difficult to see that each such $S(g, (1))$ has a unique extension 
$\widetilde{S(g, (1))}$ to $K^\varepsilon$ such that it restricts trivially to $C_2$. Thus, 
observe that if $\psi\in\mathrm{Irr}(K^\eps)$ then $\psi$ is of one of the following forms:
\begin{eqnarray*}
\psi=
\begin{cases}
\widetilde{S(g,(1))}\cdot\chi^{\nu}  &\text{for some }\ \nu\in\mathcal{P}(2),\  \ \text{ if } g\in G_1^{\eps }\,,\\

S(g,(1))\big\uparrow_{\GL_1(q^2)}^K &\text{ if } g\in \mathbb{F}_{q^2}^{\times}\smallsetminus G_1^{\eps}\,.\\
\end{cases}
\end{eqnarray*}

\bigskip

\noindent\textbf{Construction of the bijection.}
Let $\chi=\theta_1\otimes\theta_2\otimes\cdots\otimes\theta_r\big\uparrow_{I_\rho}^
{H^\eps}\in\mathrm{Irr}_{3'}(H^{\eps})$ be as above.
Let $\Omega_2$ be the subset of $\mathbb{F}_{q^2}^{\times}$ defined by 
$$\Omega_2=\{g\in\mathbb{F}_{q^2}^{\times}\smallsetminus G_1^{\eps}\ |\ \psi_j=S(g,(1))\big\uparrow_{\GL_1(q^2)}^{K^\eps}, \ \exists\ j\in\{1,\ldots, r\}\}.$$
Similarly we let $\Omega_1=\Omega_1^{(2)}\cup\Omega_1^{(1,1)}$, where for all $\nu\in\mathcal{P}(2)$ we define $\Omega_1^\nu$ as
$$\Omega_1^{\nu}=\{g\in G_1^{\eps}\ |\ \psi_j=\widetilde{S(g,(1))}\cdot\chi^{\nu}, \ \exists\ j\in\{1,\ldots, r\}\}.$$
Of course, here we are using the same notation 
as above for the extension $\widetilde{S(g,(1))}$, for $g\in G_1^{\eps}$.
Observe that the definition of all the sets above depends on $\eps$ and on the irreducible character $\chi$. These dependences are omitted in our notation. 
Moreover, the set $\Omega_2$ is regarded as a set of representatives for the equivalence classes $[g]=\{g,g^{\eps q}\}$. It is also important to observe that the sets $\Omega_1^{(2)}$ and $\Omega_1^{(1,1)}$ are not necessarily disjoint. 

For $g\in\Omega_2$ let $M_2(g)=\{j\in\{1,\ldots, r\}\ |\ \psi_j=S(g,(1))\big\uparrow_{\GL_1(q^2)}^{K^\eps}\}$. Similarly for $\nu\in\mathcal{P}(2)$ and $g\in\Omega_1^{\nu}$ define 
$M_1^\nu(g)=\{j\in\{1,\ldots, r\}\ |\ \psi_j=\widetilde{S(g,(1))}\cdot\chi^{\nu}\}$. Let $M_1(g)=M_1^{(2)}(g)\cup M_1^{(1,1)}(g)$.
Since $\psi_1,\ldots, \psi_r$ are pairwise distinct, we deduce that $|M_2(g)|=1$ for all $g\in\Omega_2$. Hence, we can relabel some of our variables in the following way. For $g\in\Omega_2$, if $M_2(g)=\{j\}$ then denote $m_j$ by $m_g$ and $\lambda_j$ by $\lambda_g$. Clearly we now have that $\lambda_g\vdash m_g$. 
Similarly, for $g\in\Omega_1$, we observe that $1\leq |M_1(g)|\leq 2$ and $|M_1(g)|=2$ if and only if $g\in\Omega_1^{(2)}\cap\Omega_1^{(1,1)}$. 
Hence we can proceed as follows.

If $g\in\Omega_1\smallsetminus(\Omega_1^{(2)}\cap\Omega_1^{(1,1)})$, then $M_1(g)=\{j\}=M_1^{\nu}$ for some $j\in\{1,\ldots, r\}$ and for a unique $\nu\in\mathcal{P}(2)$. In this case we denote $m_j$ by $m_g^\nu$ and $\lambda_j$ by $\lambda_g^\nu$. Moreover we let $m_g^{\nu'}=0$ and $\lambda_g^{\nu'}=\emptyset$.
(Recall that $\nu'$ denotes the conjugate partition of $\nu$).

On the other hand, if $g\in\Omega_1^{(2)}\cap\Omega_1^{(1,1)}$ then there exist distinct $i,j\in\{1,\ldots, r\}$ such that $M_1^{(2)}(g)=\{i\}$ and $M_1^{(1,1)}(g)=\{j\}$. In this case we denote $m_i$ by $m_g^{(2)}$, 
$\lambda_i$ by $\lambda_g^{(2)}$, $m_j$ by $m_g^{(1,1)}$ and  $\lambda_j$ by $\lambda_g^{(1,1)}$. Moreover we let $m_g=m_g^{(2)}+m_g^{(1,1)}$.

\begin{remark}
The relabelling of the integers $m_1,\ldots, m_r$ and of the partitions $\lambda_1,\ldots,\lambda_r$ described above is well defined because the sets $$M_2(g), M_1^{(2)}(g), M_1^{(1,1)}(g)      \ \ \ \text{for all}\ \ \ \ g\in\Omega_1\cup\Omega_2,$$
are pairwise disjoint. 
Moreover, the relabelling is natural and choice-free since each of those sets has size at most $1$ (so no choice is made). 
\end{remark}

With this in mind, we let $\chi^\star$ be the irreducible character of $G_n^\eps$ defined as the following circle product of primary irreducible characters. 
$$\chi^\star=\prod_{g\in\Omega_2}^{\circ}S(g,\lambda_g)\circ\prod_{h\in\Omega_1}^{\circ}S(h,\zeta(h)),$$
where $\zeta(h)$ is the unique partition of $2m_h$ with $2$-core equal to $\emptyset$ and $2$-quotient equal to $(\lambda_h^{(2)},\lambda_h^{(1,1)})$.

\begin{thm}\label{thm q+1}
The map $\chi\mapsto\chi^\star$ is a canonical bijection between
$\Irr_{3'}(H^\eps)$ and $\Irr_{3'}(G_n^\eps)$.
Moreover, the bijection is both $\Gamma$-equivariant and $D^\eps$-equivariant. 
\end{thm}
\begin{proof}
We keep the notation introduced in Section \ref{q+1}. Let $\Omega=\Omega_1\cup\Omega_2$.
Let $k=|\Omega|$. Then $\chi^\star=S_1\circ\cdots\circ S_k$, where for all $\ell\in\{1,\ldots, k\}$ we have that $S_\ell$ is a primary irreducible $3'$-character of $G_{2m_g}^\eps$, for some $g\in\Omega$.
This follows by observing that for $g\in\Omega_2$ we have that $3\nmid\chi^{\lambda_g}(1)$. On the other hand, for $h\in\Omega_1$ we have that $\zeta(h)$ has $2$-core equal to $\emptyset$ and $2$-quotient equal to $(\lambda_1, \lambda_2)$, where $$3\nmid\chi^{\lambda_j}(1) \ \ \ \text{and} \ \ \ 3\nmid\frac{m_h!}{\prod_{j\in\{1,2\}}|\lambda_j|!}.$$
Hence in both cases $S_\ell$ satisfies condition (2) of \cite[Lemma 5.4]{Olsson} (for $\eps=+1$) or the hypothesis of Theorem \ref{char-gu} (for $\eps=-1$).

It is easy to see that the partition underlying the sequence $(m_g)_{g\in\Omega}$ lies in $\Lambda_{3'}(m)$. For example notice that 
$$\fS_{m_1}\times\cdots\times \fS_{m_r}\leq \mathop{\times}\limits_{g\in\Omega}\fS_{m_g}\leq \fS_m.$$
Hence for $\eps=+1$ and for $\eps=-1$, we have that $\chi^\star\in\mathrm{Irr}_{3'}(G_n^\eps)$, respectively by \cite[Lemma 5.3]{Olsson} and by Theorem \ref{char-gu}. 

The map is a bijection since every step in its definition is choice-free, unique and reversible. It is therefore easy to construct its inverse. 
Equivariance with respect to the actions of $\Gamma$ and $D^\eps$ follows again from Lemma \ref{wr-ext} together with Remark \ref{rem:actionaut} (recall that both $\Gamma$ and $D^\eps$ fix every permutation matrix\color{black}). 
\end{proof}

\smallskip
The passage from $2m$ to $2m+1$ is given in the following statement.

\begin{thm}\label{odd-even}
Let $m \geq 1$ be an integer, $\eps = \pm$, and let $q \equiv -\eps (\mod 3)$ be a prime power. 
\begin{enumerate}[\rm(i)]
\item If we embed $G^\eps_1(q) \times G^\eps_{2m}(q)$ as a standard subgroup of $G^\eps_{2m+1}(q)$ and choose a 
Sylow $3$-subgroup $P$ of the direct factor $G^\eps_{2m}(q)$, then $P \in \Syl_3(G^\eps_{2m+1}(q))$ and 
$$\NB_{G^\eps_{2m+1}(q)}(P) = G^\eps_1(q) \times \NB_{G^\eps_{2m}(q)}(P).$$ 

\item There is a canonical bijection between 
$\Irr_{3'}(G^\eps_{2m+1}(q))$ and $\Irr(G^\eps_1(q)) \times \Irr_{3'}(G^\eps_{2m}(q))$. 
\end{enumerate}
\end{thm}

\begin{proof}
(a) To prove (i), it suffices to note that $3 \nmid [G^\eps_{2m+1}(q):G^\eps_{2m}(q)]$, and that all composition factors of 
$P$ on the natural module $U =\FF^{2m}$ of $G^\eps_{2m}(q) = G^\eps(U)$ are of even dimension, where 
$\FF= \FF_q$ if $\eps = +$ and $\FF = \FF_{q^2}$ if $\eps = -$. Hence $P \in \Syl_3(G^\eps(V))$
for $V := L \oplus U$ with $L := \FF$, and $\NB_{G^\eps(V)}(P)$ preserves this decomposition of $V$.

\smallskip
(b) Express any $\chi \in \Irr_{3'}(G^\eps_{2m+1}(q))$ in the form \eqref{gl1} or \eqref{gu1}, and apply \cite[Lemma (5.3)]{Olsson} in 
the case $\eps = +$. In the case $\eps = -$, apply Theorem \ref{char-gu}. Then there is 
a unique $i$ such that $k_id_i$ is odd. Relabeling the $k_jd_j$ if necessary, we may assume that $2 \nmid k_1d_1$. Next, applying
\cite[Lemma (5.4)]{Olsson} and Theorem \ref{char-gu} to $S(s_1,\mu) \in \Irr_{3'}(G^\eps_{k_1d_1}(q))$ with $\mu := \lam_1$, we see that
$d_1=1$, $s_1^{q-\eps}=1$, $\mu_{(2)} = (1)$, $\mu^{(2)} = (\mu_0,\mu_1)$, $\mu_i$ is a $3'$-partition of $m_i$ for $i = 0,1$,
$m_0+m_1 = (k_1d_1-1)/2$, and $3 \nmid \binom{(k_1d_1-1)/2}{m_1}$. Now by \cite[5.16]{R} 
there is a unique $\nu \vdash (k_1d_1-1)$ such that $\nu_{(2)} = \emptyset$ and $\nu^{(2)} = \mu^{(2)}$. Applying 
\cite[Lemma (5.4)]{Olsson} and Theorem \ref{char-gu} again, we see that $S(s_1,\nu) \in \Irr_{3'}(G^\eps_{k_1d_1-1}(q))$. It follows by the 
parametrization of $\Irr(\GL_{2m}(q))$ given in \eqref{gl1} and of $\Irr(\GU_{2m}(q))$ given in \eqref{gu1} that
\begin{equation}\label{gl2}
  \chi^* := S(s_1,\nu) \circ S(s_2,\lam_2) \circ \ldots \circ S(s_m,\lam_m) \in \Irr_{3'}(G^\eps_{2m}(q)),
\end{equation}  
where the fact that $3 \nmid \chi^*(1)$ follows from \cite[Lemma (5.3)]{Olsson} and Theorem \ref{char-gu}.
Note that the term $S(s_1,\nu)$ in \eqref{gl2} disappears if and only if $k_1d_1=1$.

Now we can define the following canonical map
$$\Theta: \Irr_{3'}(G^\eps_{2m+1}(q)) \to \Irr(G^\eps_1(q)) \times \Irr_{3'}(G^\eps_{2m}(q)),~~\chi \mapsto (S(s_1,(1)),\chi^*).$$

\smallskip
(c) Conversely, consider any pair $(S(t,(1)),\chi^*) \in  \Irr(G^\eps_1(q)) \times \Irr_{3'}(G^\eps_{2m}(q))$ with 
$\chi^*$ written in the form \eqref{gl2}; in particular, $t^{q-\eps}=1$. If $t$ is different from all $s_i$ in \eqref{gl2},
then we define 
$$\al := S(t,(1)) \circ S(s_1,\nu) \circ S(s_2,\lam_2) \circ \ldots \circ S(s_m,\lam_m) \in \Irr(G^\eps_{2m+1}(q)).$$
Note that $\al(1) = (q^{2m+1}-\eps)\chi^*(1)$ is coprime to $3$, and $\Theta(\al) = (S(t,(1)),\chi^*)$. 

In the remaining case, there is precisely one $s_i$ in \eqref{gl2} that is equal to $t$; without loss we may assume 
that $s_1 = t$. By \cite[Lemma (5.3)]{Olsson} and Theorem \ref{char-gu} applied to $\chi^*$, $m':= \deg(s_1)|\nu|$ is even, and 
$3 \nmid \deg S(s_1,\nu)$. By \cite[Lemma (5.4)]{Olsson} and Theorem \ref{char-gu} applied to $S(s_1,\nu)$,  
$\nu_{(2)} = \emptyset$, $\nu^{(2)} = (\nu_0,\nu_1)$, $\nu_i$ is a $3'$-partition of $m'_i$ for $i = 0,1$,
$m'_0+m'_1 = m'/2$, and $3 \nmid \binom{m'/2}{m'_1}$. Again by \cite[5.16]{R}, there is a unique 
$\gam \vdash (m'+1)$ such that $\gam_{(2)} = (1)$ and $\gam^{(2)} = \nu^{(2)}$. Now
$$\beta := S(s_1,\gam) \circ S(s_2,\lam_2) \circ \ldots \circ S(s_m,\lam_m) \in \Irr(G^\eps_{2m+1}(q)).$$
Using \cite[Lemma (5.4)]{Olsson} and Theorem \ref{char-gu}
we see that $3 \nmid \deg S(s_1,\gam)$. Then applying \cite[Lemma (5.3)]{Olsson} and Theorem \ref{char-gu} 
to both $\chi^*$ and $\beta$,
we can conclude that $3 \nmid \beta(1)$, and that $\Theta(\beta) =  (S(t,(1)),\chi^*)$.

Thus $\Theta$ is surjective. On the other hand, 
since the McKay conjecture holds for $G^\eps_k(q)$ for any $k$ (see \cite{FS}), (i) implies that 
two sets  $\Irr(G^\eps_1(q)) \times \Irr_{3'}(G^\eps_{2m}(q))$ and $\Irr_{3'}(G^\eps_{2m+1}(q))$
have the same cardinality. 
Consequently, $\Theta$ is a bijection, and we have proved (ii).
\end{proof}

\section{Linear and Unitary Groups: A Local Bijection}
As usual, we say that a partition $\lam \vdash m$ is a {\it $3'$-partition of $m$} (and write $\lam \vt m$) if and only if $3 \nmid \chi^\lam(1)$.

We will make use of the following hypothesis, which is fulfilled by the results of \S\S3, 4:

(*) {\it For every $m = \sum^t_{i=0}m_i$ with $m_i = a_i3^i$, $0 \leq a_i \leq 2$, we have constructed a canonical bijection
$$\Phi_m:\Irr_{3'}(\fS_m) \to \Irr_{3'}(\NB_{\fS_m}(R_m)),~~\chi^\lam \mapsto \lam^\sharp,$$
which is compatible with breaking $m$ into its $3$-adic pieces $m_i$ in the sense of Theorem \ref{thm:maintotal}}, when we choose
$R_m \in \Syl_3(\fS_m)$ to be $R_m = R_{m_1} \times \ldots \times R_{m_t}$ with 
$R_{m_i} \in \Syl_3(\fS_{m_i})$ and $\fS_{m_1} \times \ldots \times \fS_{m_t}$ is naturally
embedded in $\fS_m$ as a Young subgroup.

If $R \leq \fS_m$, then let $i^R$ denote the $R$-orbit of $1 \leq i \leq m$.
The proof of the following statement is straightforward and will be omitted:

\begin{lemma}\label{norm}
Let $p$ be a prime, and let $K$ be a finite group with $Q \in \Syl_p(K)$. Let $H := K \wr \fS_m$, and 
$R \in \Syl_p(\fS_m)$ so that $P := Q \wr \fS_m \in \Syl_p(H)$. Then 
$$\NB_H(P) = \{ (x_1, \ldots,x_m; z) \mid x_i \in \NB_K(Q),~z \in \NB_{\fS_m}(R),~x_{i}^{-1}x_j \in Q
   \mbox{ if }i^R = j^R\}.$$  
\end{lemma}

\begin{thm}\label{local}
Let $K$ be a finite group with a normal, abelian $Q \in \Syl_3(K)$. Fix any 
$m = \sum^t_{j=0}m_j \geq 1$ with $m_j = a_j3^j$, $0 \leq a_j \leq 2$,
and $R \in \Syl_3(\fS_m)$. Let $H = K \wr \fS_m$ and let $P = Q \wr R \in \Syl_3(H)$. Then there is 
a canonical bijection
$$\Theta:\Irr_{3'}(H) \to \Irr_{3'}(\NB_H(P)).$$
In particular, $\Theta$ is $\Gal(\bar\QQ/\QQ)$-equivariant. Moreover, if a finite group $A$ acts on $K$ and 
we extend its action to $H$ by letting $A$ act trivially on $\fS_m$, then $\Theta$ is $A$-equivariant.
\end{thm}

\begin{proof}
(a) First we fix a total order on $\Irr_{3'}(K)$.
For any $\sm \in \Irr_{3'}(K)$, the character $\sm^{\otimes m} \in \Irr(K^m)$ has a canonical extension 
$\tilde\sm$ to $H$ defined by Lemma \ref{wr-ext}. Next, any $3'$-partition $\lam \vdash m$ yields
a canonical character $\tilde\sm \cdot \chi^\lam$ of $3'$-degree of 
$K^m \wr \fS_m$, where $\chi^\lam \in \Irr(\fS_m)$ corresponds to $\lam$.

Now it is easy to see that every $\chi \in \Irr_{3'}(H)$ is uniquely labeled by 
\begin{equation}\label{chi1}
  \chi = \chi(\rho,\usm,\ulam),
\end{equation}  
where $\rho = (k_1 \geq k_2 \geq \ldots \geq k_r \geq 1) \vdash m$, $\usm = (\sm_1, \ldots,\sm_r)$,
$\sm_i \in \Irr_{3'}(K)$ are pairwise distinct, $\ulam = (\lam_1, \ldots,\lam_r)$, $\lam_i$ is a $3'$-partition of $k_i$,
$3 \nmid (m!/\prod^r_{i=1}k_i!)$, and if $k_i = k_{i'}$ for $i < i'$ then $\sm_i > \sm_{i'}$. Indeed, given such 
a $(\rho,\usm,\ulam)$, we can consider the character 
$$\phi := \sm_1^{\otimes k_1} \otimes \sm_2^{\otimes k_2} \otimes \ldots \otimes \sm_r^{\otimes k_r}$$
of $K^m$, which has the inertia subgroup $K^m \rtimes Y$ in $H$, where
$$Y = \fS_{k_1}  \times \fS_{k_2} \times \ldots \times \fS_{k_r}.$$
By Lemma \ref{wr-ext}, $\sm_i^{\otimes k_i}$ has a canonical 
(equivariant under $A$ and $\Gamma$) extension $\tilde\sm_i$ to $K^{k_i} \wr \fS_{k_i}$, and
then we get the irreducible character 
$$\tilde\phi := (\tilde\sm_1 \cdot \chi^{\lam_i}) \otimes \ldots \otimes (\tilde\sm_r \cdot \chi^{\lam_r})$$
of $K^m \rtimes Y$. Inducing $\tilde\phi$ to $H$, we obtain $\chi$.

The condition 
$3 \nmid (m!/\prod^r_{i=1}k_i!)$ is equivalent to
\begin{equation}\label{chi2}
  k_i = \sum^t_{j=0}b_{i,j}3^j,~1 \leq i \leq r,~~\sum^r_{i=1}b_{i,j} = a_j,~0 \leq j \leq t.
\end{equation}  
We can further refine the $i^{\mathrm {th}}$ component $\lam_i$ of the parameter $\ulam$ of $\chi$ as follows. 
Recall $\lam_i \vt k_i = \sum^t_{j=0}k_{i,j}$,
where $k_{i,j} := b_{i,j}3^j$. Choosing $R_{k_{i,j}} = R_{3^j}^{b_{i,j}}$, we have 
$R_{k_i} = R_{k_{i,0}} \times R_{k_{i,1}} \times  \ldots \times R_{k_{i,t}} \in \Syl_3(\fS_{k_i})$, and 
$$\NB_{\fS_{k_i}}(R_{k_i}) = \NB_{\fS_{k_{i,0}}}(R_{k_{i,0}}) \times  \NB_{\fS_{k_{i,1}}}(R_{k_{i,1}}) \times \ldots \times 
    \NB_{\fS_{k_{i,0}}}(R_{k_{i,0}}).$$
Hence, by (*) (which follows from Theorem \ref{thm:maintotal}), we can find a unique $\lam_{i,j} \vt k_{i,j}$ for each $j$ such that
$$(\lam_i)^\sharp = (\lam_{i,0})^\sharp \otimes (\lam_{i,1})^\sharp \otimes \ldots \otimes (\lam_{i,t})^\sharp.$$  
Thus each $\lam_i$ with $1 \leq i \leq r$ is uniquely determined by the tuple $(\lam_{i,0},\lam_{i,1}, \ldots,\lam_{i,t})$.  

\smallskip
(b) For each $i$, we fix $R_{3^i} \in \Syl_3(\fS_{3^i})$. Then we can take $R_{m_i} \in \Syl_3(\fS_{m_i})$ to be 
$R_{3^i}^{a_i}$, and then take $R$ to be $R_{m_0} \times \ldots \times R_{m_t}$.
Note that $N:=\NB_H(P)$ normalizes $Q^m$ and the subgroup
$[Q^m,P] = [Q^m,R]$ of $[P,P]$. Modding out by this subgroup and using the explicit structure of $N$ given 
in Lemma \ref{norm}, 
we can identify $\Irr_{3'}(N)$ with $\Irr_{3'}(M)$, where 
\begin{equation}\label{m1}
  M:= N/[Q^m,P] \cong M_0 \times M_1 \times \ldots \times M_t.
\end{equation}  
Here, 
\begin{equation}\label{m2}
  M_i \cong K^{a_i} \rtimes \NB_{\fS_{m_i}}(R_{m_i}) \cong (K \times \NB_{\fS_{3^i}}(R_{3^i})) \wr \fS_{a_i},
\end{equation}
and we use the convention $M_i = 1$ if $a_i=0$.   
(Indeed, if $m = 3^t$ for instance, then by Lemma \ref{norm}, $h = (h_1, \ldots,h_m;z)$ belongs to $N \cap K^m$ if and only if
$z=1$, $h_i = x_iy$ for all $1 \leq i \leq m$, $x_i \in Q$ and $y \in T$, where we have expressed $K = Q \rtimes T$ for a fixed
$3'$-subgroup $T$ using the Schur-Zassenhaus theorem.  Now the map 
$$(x_1y,x_2y, \ldots,x_my;1) \mapsto x_1x_2\ldots x_my$$
yields an isomorphism $(N \cap K^m)/[Q^m,P] \cong K$. Furthermore, the complement $\NB_{\fS_m}(R_m)$ 
in $N$ centralizes $(N \cap K^m)/[Q^m,P]$. The general case then follows by a straightforward consideration.)

Now, using \eqref{m1}, we can represent each $\theta \in \Irr_{3'}(N) = \Irr_{3'}(M)$ uniquely as 
\begin{equation}\label{theta1}
  \theta = \theta_0 \otimes \theta_1 \otimes \ldots \otimes \theta_t,
\end{equation}  
where $\theta_i \in \Irr_{3'}(M_i)$.

\smallskip
(c) Next we parametrize $\Irr_{3'}(M_i)$. First we consider the case $a_i = 1$. Then by \eqref{m2} we can identify $M_i$ with 
$K \times \NB_{\fS_{3^i}}(R_{3^i})$. Hence, by Theorem A, 
each $\theta_i \in \Irr_{3'}(M_i)$ is uniquely written as $\tau_i \otimes \mu_i^\sharp$, where 
$\tau_i \in \Irr_{3'}(K)$ and $\mu_i \vt 3^i$ (equivalently, $\mu_i \in \HC(3^i)$, so that $\Phi_{3^i}(\chi^{\mu_i}) = \mu_i^\sharp$). Thus 
when $a_i=1$, there is a canonical bijection between $\Irr_{3'}(M_i)$  and $\Irr_{3'}(K) \times \Irr_{3'}(\fS_{3^i})$. 

\smallskip
(d) Next assume that $a_i = 2$. Using \eqref{m2} we identify $M_i$ with $(K \times X) \wr \fS_2$, where
$X := \NB_{\fS_{3^i}}(R_{3^i})$.  Consider any $\theta_i \in \Irr_{3'}(M_i)$. 

\smallskip
(d1) First assume that an irreducible constituent of 
the restriction of $\theta_i$ to $K \times K \lhd M_i$ is different on restriction to the two direct factors. 
In this case, $\theta_i$ is induced from the character
$$\tau^1_i \otimes (\nu^1_i)^\sharp \otimes \tau^2_i \otimes (\nu^2_i)^\sharp$$
of $K \times X \times K \times X$, where $\tau^{1,2}_i \in \Irr_{3'}(K)$, $\tau^1_i > \tau^2_i$, and $\nu^{1,2}_i \in \HC(3^i)$. 
Thus $\theta_i$ is uniquely determined by a $2$-set $\{\tau^1_i,\tau^2_i\}$ and an {\it ordered} pair $(\nu^1_i,\nu^2_i)$.

\smallskip
(d2) In the remaining case we have $\tau_i \otimes \tau_i$ as an irreducible constituent of $(\theta_i)|_{K \times K}$ for 
some $\tau_i \in \Irr_{3'}(K)$. Then we have two possibilities. In the former, 
$\theta_i$ is induced from the character
$$\tau_i \otimes \beta^1_i \otimes \tau_i \otimes \beta^2_i$$
of $H \times X \times H \times X$, where $\beta^{1,2}_i \in \Irr_{3'}(X)$ are distinct. Correspondingly, there is a unique 
$\nu_i \in \Irr_{3'}(\fS_{m_i})$ such that $\nu_i^\sharp \in \Irr_{3'}(\NB_{\fS_{m_i}}(R_{m_i}))$ is induced from the character 
$\beta^1_i \otimes \beta^2_i$
of $X \times X \lhd \NB_{\fS_{m_i}}(R_{m_i}) = X \wr \fS_2$. In the latter, $\theta_i$ extends the character 
$$\tau_i \otimes \beta_i \otimes \tau_i \otimes \beta_i$$
of $K \times X \times K \times X$, where $\beta_i \in \Irr_{3'}(X)$. By Lemma \ref{wr-ext}, such an extension is
uniquely determined by the sign of the trace of the involution $(1,2) \in \fS_2$. 
Correspondingly, there are two characters $\nu_i^\pm$ in $\Irr_{3'}(\fS_{m_i})$ such that 
$(\nu_i)^\sharp \in \Irr_{3'}(\NB_{\fS_{m_i}}(R_{m_i}))$ extend the character 
$\beta_i \otimes \beta_i$
of $X \times X \lhd \NB_{\fS_{m_i}}(R_{m_i}) = X \wr \fS_2$, and they are distinguished by the sign of the trace of the involution 
$(1,2) \in \fS_2$. Hence there is a unique $\nu_i \vt m_i$ such that $\nu_i^\sharp$ agrees with $\theta_i$ on 
the sign of the trace of $(12)$. 
Thus, in either one of the two possibilities, $\theta_i$ is uniquely determined by $\tau_i \in \Irr_{3'}(K)$ and $\nu_i \vt m_i$.

We have shown that, in the case $a_i=2$, there is a canonical bijection between $\Irr_{3'}(M_i)$ and the (disjoint) union of 
$\Irr_{3'}(K) \times \Irr_{3'}(\fS_{m_i})$ and $\Irr_{3'}(K)^{\{2\}} \times \Irr_{3'}(\fS_{3^i})^{2}$. (Here,
for any finite set $\Omega$, we let $\Omega^{\{2\}}$ denote the set of all $2$-subsets of $\Omega$, and
$\Omega^2$ denote the Cartesian product $\Omega \times \Omega$.)

\smallskip
(e) Consider any $\chi = \chi(\rho,\usm,\ulam) \in \Irr_{3'}(H)$ as in \eqref{chi1}.
Recall that $\rho = (k_1 \geq k_2 \geq \ldots \geq k_r) \vdash m$ satisfies \eqref{chi2}. Consider any $0 \leq j \leq t$ with $a_j > 0$
(note that $M_j = 1$ if $a_j=0$). Then there is some 
$k_i = \sum^t_{j=0}k_{i,j}$ with $1 \leq i \leq r$ and $b_{i,j} > 0$. 
Suppose that $b_{i,j} = a_j$. Then we choose $\theta_j \in \Irr_{3'}(M_j)$ labeled by 
$\sigma_i \in \Irr_{3'}(K)$ and $\lam_{i,j} \vt k_{i,j} = m_j$.

Suppose now that $b_{i,j} \neq a_j$. Then \eqref{chi2} implies that $(b_{i,j},a_j) = (1,2)$, and that 
there is a unique $i' \neq i$ such that $b_{i',j} = 1$. We may assume for definiteness that $i < i'$, whence $\sigma_i > \sigma_{i'}$
by our construction of $(\rho,\usm,\ulam)$. Then we choose $\theta_j \in \Irr_{3'}(M_j)$ labeled by the $2$-set
$\{\sigma_i,\sigma_{i'}\} \in \Irr_{3'}(K)^{\{2\}}$ and the ordered pair $(\lam_{i,j},\lam_{i',j})$ of $3^j$-hooks as in (d1).

Thus we have assigned to $\chi$ a canonical $\theta_j \in \Irr_{3'}(M_j)$ for each $0 \leq j \leq t$. It is straightforward to check 
that the map $\chi \mapsto \chi^\sharp$, with $\chi^\sharp := \theta_0 \otimes \theta_1 \otimes \ldots \otimes \theta_t \in \Irr_{3'}(N)$
as in \eqref{theta1}, is a canonical bijection. 

To conclude let $\Gamma$ denote either the absolute Galois group $\Gal(\bar\QQ/\QQ)$ or the group of automorphisms $A$ described in the statement of the Theorem. It is easy to observe that for $g\in\Gamma$ and $\chi=\chi(\rho,\usm,\ulam)\in\mathrm{Irr}_{3'}(H)$, we have that $\chi^g=\chi(\rho,\usm^g,\ulam)$, where $\usm^g=(\sigma_1^g,\ldots, \sigma_r^g)$. This follows from direct computations, using Lemma \ref{wr-ext} and the fact that the characters of the symmetric groups are rational-valued, and so $\Gamma$-invariant\color{black}. 
Moreover, if $\theta\in\mathrm{Irr}_{3'}(M_i)$ for some $i$ such that $a_i=1$, then $\theta=\theta(\sigma, \mu)$ is labelled by $\sigma, \mu\in\mathrm{Irr}_{3'}(K)\times \mathrm{Irr}_{3'}(\fS_{3^{i}})$. It is again not difficult to see that $\theta^g=\theta(\sigma^g,\mu)$, for all $g\in\Gamma$, 
using that the characters in $\irrq{\norm{\fS_{3^i}}{P_{3^i}}}{3}$ are rational-valued and thus $\Phi_{3^i}$ is trivially
$\Gamma$-equivariant.
A similar observation shows that also for $a_i=2$ we have that $\Gamma$ acts non-trivially only on the $\mathrm{Irr}_{3'}(K)$-part of the label of an irreducible $3'$-degree character of $M_i$. 
From this we deduce that the canonical bijection $\Theta$ is both $\Gal(\bar\QQ/\QQ)$-equivariant and $A$-equivariant. 
\end{proof}

\section{Proofs of Theorem C and Corollary D}

Let $n=2m+c$, for some $c\in\{0,1\}$. Let $P$ be a Sylow $3$-subgroup of $G_n^\eps(q)$, chosen (as in Theorem \ref{odd-even})
such that $\NB_{G^\eps_{2m+1}(q)}(P) = G^\eps_1(q) \times \NB_{G^\eps_{2m}(q)}(P).$
Let $H^\eps(q)$ be the subgroup of $G_{2m}^\eps(q)$ defined in Sections \ref{q-1} and \ref{q+1} by 
\begin{eqnarray*}
H^\eps(q)=
\begin{cases}
G^\eps_{1}(q)\wr \fS_n,\  \ &\text{ if } 3\ \text{divides}\ q-\eps\,,\\

(\GL_1(q^2)\rtimes_{\eps} C_2)\wr \fS_m,\  \  &\text{ if } 3\ \text{divides}\ q+\eps\,.\\
\end{cases}
\end{eqnarray*}

It follows that $\NB_{G^\eps_{n}(q)}(P)\leq G^\eps_c(q) \times H^\eps(q)\leq G^\eps_c(q)\times G^\eps_{2m}(q)\leq G^\eps_n(q)$ (see for instance \cite[Section 3]{FS}.
This in particular implies that $\NB_{G^\eps_{2m}(q)}(P)=\NB_{H^\eps(q)}(P)$.
Therefore Theorem \ref{local} gives a canonical bijection $\Theta$ between $\mathrm{Irr}_{3'}(H^\eps(q))$ and $\mathrm{Irr}_{3'}(\NB_{G^\eps_{2m}(q)}(P))$. 

Denote by $\Phi_{q,\eps}$ the map between 
$\mathrm{Irr}_{3'}(G^\eps_{2m}(q))$ and $\mathrm{Irr}_{3'}(H^\eps(q))$, 
described in Proposition \ref{prop q-1} (when $3$ divides $q-\eps$) or in Theorem \ref{thm q+1} (when $3$ divides $q+\eps$).
If $c=0$ then $\Theta\circ\Phi_{q,\eps}$ is a canonical bijection between $\mathrm{Irr}_{3'}(G^\eps_{n}(q))$ and $\mathrm{Irr}_{3'}(\NB_{G^\eps_{n}(q)}(P))$.

When $c=1$, let $\Psi$ be the map described in Theorem \ref{odd-even} (ii). In this case we have that $(\mathrm{id}\times \Theta)\circ (\mathrm{id}\times \Phi_{q,\eps})\circ\Psi$ is a canonical bijection between $\mathrm{Irr}_{3'}(G^\eps_{n}(q))$ and $\mathrm{Irr}_{3'}(\NB_{G^\eps_{n}(q)}(P))$.

In both cases the bijection obtained is equivariant with respect to the action of the absolute Galois group $\Gal(\bar\QQ/\QQ)$ and with respect to the action of group automorphisms because it is obtained as the composition of equivariant bijections. 

This completes the proof of Theorem C. Corollary D directly follows from Theorem C.

\end{document}